\documentclass[a4paper]{article}

\usepackage{
amsmath,
amsthm,
amscd,
amssymb,
stmaryrd,
wasysym
}
\usepackage{comment}
\usepackage{tikz}
\usetikzlibrary{positioning,arrows,calc}
\usepackage{xspace}

\usepackage{graphicx}

\usepackage{url}

\theoremstyle{plain}
\newtheorem{lemma}{Lemma}
\newtheorem{definition}{Definition}
\newtheorem{corollary}{Corollary}

\newtheorem{theorem}{Theorem}

\newtheorem{question}{Question}
\newtheorem{remark}{Remark}

\newtheoremstyle{derp}
{3pt}
{3pt}
{}
{}
{\upshape}
{:}
{.5em}
{}
\theoremstyle{derp}
\newtheorem{example}{Example}

\newcommand{\Z}{\mathbb{Z}}

\newcommand{\N}{\mathbb{N}}

\newcommand\xqed[1]{%
  \leavevmode\unskip\penalty9999 \hbox{}\nobreak\hfill
  \quad\hbox{#1}}
\newcommand\qee{\xqed{$\fullmoon$}}

\newcommand{\Aut}{\mathrm{Aut}}

\newcommand{\llb}{\llbracket}
\newcommand{\rrb}{\rrbracket}

\newcommand{\Alt}{\mathrm{Alt}}
\newcommand{\Sym}{\mathrm{Sym}}

\title{Periodic points and residual finiteness of automorphism groups of subshifts}

\author{
Ville Salo \\
vosalo@utu.fi
}

\begin{document}
\maketitle

\begin{abstract}
If totally periodic points are dense in a subshift $X$, its automorphism group is residually finite. We show a weak converse: if periodic points are not dense in a subshift $X$, then the automorphism group of $X \times Y$ is not residually finite for full shifts $Y$ (and sufficiently full-shift-like subshifts). On the other hand, we show that the automorphism group of a block gluing $\Z^2$-subshift is always locally embeddable in finite groups (thus sofic). Hochman recently constructed a strongly irreducible $\Z^2$-subshift with no periodic points. Combining our result with this example gives a strongly irreducible $\Z^2$-subshift whose automorphism group is not residually finite, which solves a question of Coornaert and Ceccherini-Silberstein. 
\end{abstract}

\section{Introduction}

Let $A$ be a finite set called an \emph{alphabet} (always with at least two elements, called \emph{symbols}), and $G$ a countably infinite group. Then $A^G$ under the product topology is homeomorphic to the Cantor set, and $G$ acts on it by $gx_h = x_{hg}$. This system is called the \emph{full shift}, and a \emph{subshift} is $X \subset A^G$ which is topologically closed and satisfies $gX = X$ for all $g \in G$. We denote by $\Aut(X)$ the \emph{automorphism group} of $X$, meaning the group of $G$-commuting homeomorphisms $f : X \to X$ under function composition.

The automorphism groups of subshifts have been studied in great detail in recent years. For the most part, research has concentrated on subshifts on the group $G = \Z$. One can roughly split the research into automorphism groups of ``full-shift-like'' subshifts \cite{He69,BoLiRu88,KiRo90,Sa19b,FrScTa19,Sa22b,CaSa24} and ``minimal-like/aperiodic subshifts'' \cite{SaTo15d,CyKr16b,CyFrKrPe18,PaSc23a} (although naturally not all papers belong obviously to one side).

There are also some works on automorphism groups of subshifts on other groups. The first was Hochman's paper \cite{Ho10} on automorphism groups of $\Z^d$-subshifts. The recent works \cite{BaCaRi25,Sa25b} in turn study automorphisms of $G$-subshifts for more general $G$. The latter papers show that when $X \subset A^G$ is sufficiently full-shift-like, $\Aut(X)$ reflects some geometric and algebraic properties of the shift group $G$ in interesting ways. For example, \cite{BaCaRi25} shows that $\Aut(A^{F_2})$ embeds in $\Aut(A^G)$ when $G$ is nonamenable. In turn we can state results of \cite{Sa25b} in terms of a function called the systolic growth (we omit the definition, but see \cite{BoCo16,Sa25b}). It is shown that the systolic growth of a full shift $\Aut(A^G)$ is at most $e^{e^{n^{O(1)}}}$ when $G$ is of polynomial growth, while, under the Grigorchuk gap conjecture \cite{Gr14}, for all other groups $G$ we can find a finitely-generated subgroup of $\Aut(A^G)$ for which it is $e^{e^{e^{n^{\Omega(1)}}}}$. It was also shown that $\Aut(A^{\Z^{d+1}})$ does not abstractly embed in $\Aut(A^{\Z^d})$ for any $d \geq 1$.

The systolic growth studied in \cite{Sa25b} is a quantitative property related to residual finiteness, which is a Boolean property of a group. In this paper, we study this property of $\Aut(X)$ in detail, when $X$ is not a ``full-shift-like'' subshift. We show that it roughly corresponds to the dynamical property of having dense periodic points, although we do have to first embed the subshift $X$ into a direct product (a subshift $X$ without periodic points may have trivial automorphism group).

Recall that a group $\mathcal{G}$ is \emph{residually finite} if for all $g \in \mathcal{G} \setminus \{e_{\mathcal{G}}\}$ ($e$ always denotes the identity element), there is homomorphism $\phi : \mathcal{G} \to Q$ to a finite group $Q$, such that $\phi(g) \neq e_Q$. We say a point $x \in X$ in a subshift $X \subset A^G$ is \emph{totally periodic} if its point stabilizer is of finite index in $G$. It is well-known that if $X \subset A^G$ is a subshift with dense totally periodic points, then $\Aut(X)$ is residually finite. We show a kind of converse:

\begin{theorem}
If $G$ is finitely generated and $X \subset A^G$ does not have dense totally periodic points, then the group $\Aut(X \times B^G)$ is not residually finite for $|B| \geq 2$.
\end{theorem}

To emphasize how this can be seen as an ``if and only if'' condition, note that $B^G$ certainly has dense periodic points, so one can also phrase the result as follows: if $G$ is finitely-generated, and $X \subset A^G$ can be decomposed as a direct product with a full shift (on at least two letters), then $\Aut(X)$ is residually finite if and only if $X$ has dense periodic points.

The proof is constructive, in the sense that given a clopen set which does not have totally periodic points, the proof shows how to constructs a particular element $f \in \Aut(X \times B^G)$ which cannot be mapped nontrivially to any finite quotient. In fact, given $n$, we explicitly construct a subgroup of $\Aut(X \times B^G)$ containing $f$, which is a direct product of simple groups all of which are larger than $n$.

The idea of finding large simple groups in automorphism groups (and connections to residual finiteness) is common, and plays a major role for example in \cite{Ho10,FoRuSa24,SaSc25,Sa25b}. We note that of course large simple groups in $\Aut(X)$ do not directly prevent residual finiteness -- it was already shown in \cite{He69,KiRo90} that all finite simple groups embed in $\Aut(A^\Z)$ for any alphabet $A$, but this group is residually finite.

The group $\Aut(X \times B^G)$ contains a isomorphic copy of $G$ when $|B| \geq 2$ (see Definition~\ref{def:pasi} and Lemma~\ref{lem:PartialFG}), so when $G$ is not residually finite, neither is $\Aut(X \times B^G)$ (since the property of residual finiteness is inherited by all subgroups). In other words, our result is mainly interesting when the aperiodicity arises from dynamical properties of $X$, rather than simply $G$ having few finite quotients.

Under some additional properties on $X$, one can replace the full shift $B^G$ by any subshift $Y$ where there are sufficiently many pairwise exchangeable local patterns, see Theorem~\ref{thm:NonRF} for the more general statement.

We now state a question of Coornaert and Ceccherini-Silberstein from \cite{CeCo12}.

\begin{definition}
\label{def:StronglyIrreducible}
A subshift $X \subset A^G$ is \emph{strongly irreducible} if there exists a finite subset of $S \Subset G$ such that whenever $x, y \in X$ and $M, N \subset G$ are such that $SM \cap SN = \emptyset$, there exists some $z \in X$ such that $z|M = x|M$ and $z|N = y|N$ (where $|$ denotes restriction of configurations, seen as functions).
\end{definition}

\begin{question}
Is there a strongly irreducible subshift over $\Z^2$ whose automorphism group is not residually finite?
\end{question}

It was recently shown by Hochman \cite{Ho25} that there exists a strongly irreducible $\Z^2$-subshift which does not have any periodic points. Combining this with our results solves the problem of \cite{CeCo12}:

\begin{theorem}
\label{thm:ContainsZ2}
If $G$ contains a copy of $\Z^2$, then $G$ admits a strongly irreducible subshift whose automorphism group is not residually finite.
\end{theorem}

Every strongly irreducibly $\Z$-subshift has dense periodic points \cite{Be86} (also attributed to Weiss in \cite{CeCo12}), so $\Aut(X)$ is residually finite whenever $X$ is a strongly irreducible $\Z$-subshift. We recall a proof in Lemma~\ref{lem:ZSI}.

One natural weakening of residual finiteness is \emph{local embeddability in finite groups} \cite{Ve97}. A group $\mathcal{G}$ has this property if for all $F \Subset \mathcal{G}$, there exists a finite group $Q$ and an injection $\phi : \mathcal{G} \to Q$ such that $\phi(gg') = \phi(g)\phi(g')$ whenever $g, g', gg' \in F$. This is equivalent to saying that for all balls in Cayley graphs of finitely-generated subgroups of $\mathcal{G}$, we can find (graph-)isomorphic balls in Cayley graphs of finite groups.

We show that, although strongly irreducible $\Z^2$-subshifts may not have residually finite automorphism groups, they nevertheless always have this weaker property.

\begin{theorem}
For all strongly irreducible subshifts $X \subset A^{\Z^2}$, $\Aut(X)$ is locally embeddable in finite groups.
\end{theorem}

Another further weakening is the property of \emph{soficity} of a group, see \cite{We00}. We omit the definition here, as it is lengthy, but groups that are locally embeddable in finite groups are well-known to be sofic, so we have:

\begin{corollary}
For all strongly irreducible subshifts $X \subset A^{\Z^2}$, $\Aut(X)$ is sofic.
\end{corollary}

\begin{question}
Is $\Aut(X)$ locally embeddable in finite groups, or at least sofic, for all strongly irreducible $\Z^3$-subshifts?
\end{question}

In fact, we do not know a counterexample for any acting group, apart from the trivial observation that if $G$ is not locally embeddable in finite groups, neither is $\Aut(A^G)$ (and similarly any non-sofic group would provide a non-sofic example of $\Aut(A^G)$; but at present, no group is known to be non-sofic).

A group that is not residually finite can of course have all its finitely-generated subgroups be residually finite. An extreme example is a locally finite group that is not residually finite, such as the group of all even finite-support permutations of $\N$. Our non-residual-finiteness proof is based on precisely the appearance of large alternating groups in the automorphism group, so it is of interest to realize the proof in a finitely-generated subgroup. Let us say a group is \emph{locally residually finite} if all its finitely-generated subgroups are residually finite.

\begin{theorem}
If $G$ is finitely generated and $X \subset A^G$ does not have dense periodic points the group $\Aut(X \times B^G)$ is not locally residually finite if one of the following holds:
\begin{itemize}
\item $|B| = mn$ with $m \geq 2, n \geq 3$, or
\item $|B| \geq 3$ and $X$ is $\mathcal{B}_1$-aperiodic.
\end{itemize}
\end{theorem}

Here and elsewhere, $\mathcal{B}_r$ is the ball of radius $r$ in $G$ for the chosen set of generators, and \emph{$S$-aperiodic} for $S \subset G$ means that for all $s \in S \setminus \{e\}$, $X$ has no points $x$ satisfying $sx = x$.

It seems likely that any alphabet would work in this theorem, even without the $\mathcal{B}_1$-aperiodicity assumption, but already the statements here require some technical work.

We also note the following if and only if statement on the level of groups (Theorem~\ref{thm:ContainsZ2} uses this in the case $G = \Z^2$):

\begin{corollary}
\label{cor:LevelOfGroups}
Let $G$ be a finitely generated group. Let $\mathcal{C}$ be a class of $G$-subshifts closed under direct products with full shifts. Then the following are equivalent:
\begin{itemize}
\item all subshifts in $\mathcal{C}$ have dense totally periodic points,
\item all subshifts in $\mathcal{C}$ have residually finite automorphism group,
\item all subshifts in $\mathcal{C}$ have locally residually finite automorphism group,
\end{itemize}
\end{corollary}

For example, $G$ admits a strongly irreducible subshift (SFT) without dense totally periodic points if and only if it admits a strongly irreducible subshift (SFT) with non-residually-finite automorphism group (see Definition~\ref{def:StronglyIrreducible} and Definition~\ref{def:SFT} for the definitions).



\section{Definitions and conventions}

We have $0 \in \N$. For $k \in \N$, write $\llb k \rrb = \{0, 1, \ldots, k-1\}$ and $\llb m, n\rrb = \{i \in \Z \;|\; m \leq i \leq n\}$. (So $\llb k \rrb = \llb 0, k-1 \rrb$, $\llb 1 \rrb = \emptyset$, and $\llb 0 \rrb$ is not defined.) The permutation groups $\Sym(A)$ and $\Alt(A)$ denote the groups of \emph{finite-support permutations} on the set $A$ (for the most part, $A$ is finite, but in Section~\ref{sec:Graphs} we allow larger sets as well). We write disjoint union as $A \sqcup B$; specifically, this means a union, which we know to be disjoint (so e.g.\ $\{0,1\} \sqcup \{1, 2\}$ carries no meaning and $\{0\} \sqcup \{1, 2\} = \{0\} \cup \{1, 2\}$). (Another meaning to this symbol is given below.) A \emph{word} over an alphabet $A$ is an element of the free monoid on $A$, or a function $\llb n \rrb \to A$ (where $n$ is the \emph{length} of the word, i.e.\ its word norm in the free monoid with generators $A$). For two words $u, v$, we write $u\cdot v$ or just $uv$ for their concatenation, with the obvious definition. We sometimes write $d : A \cong B$ for a bijective function from $A$ to $B$.

We write $S \Subset G$ for $S$ a finite subset of $G$. All group actions are from the left, and are covariant: $g \cdot (h \cdot x) = (gh) \cdot x$. We simply write the action of $g$ on $x$ as $gx$. By $A^B$ we formally mean the set of functions from $B$ to $A$, but we call them \emph{configurations} and use subscripts instead of function application, i.e.\ if $x \in A^B$ we write $x_b$ for $x(b)$, when $b \in B$. Restriction of configurations $x \in A^G$ to subsets $D \subset G$ is denoted $x|D$, specifically $x|D$ is the configuration $y \in A^D$ defined by $y_g = x_g$.

A \emph{pattern} is a configuration $x \in A^B$ whose domain $B$ is finite. In the case $G = \Z$, we identify words of length $n$ with patterns $\llb n \rrb \to A$ with contiguous domains. For two patterns $p \in A^D, q \in A^E$ with disjoint domains, write $t = p \sqcup q$ for the pattern in $A^{D \sqcup E}$ with $t_g = \begin{cases}p_g & \mbox{if } g \in D \\ q_g &\mbox{if } g \in E. \end{cases}$

We discuss our shift convention for a few paragraphs, as there are several conventions in the literature, and both have their strengths and weaknesses. The chosen shift convention here is $gx_h = x_{hg}$, where $g, h \in G$, $x \in A^G$. This is indeed covariant because $gg'x_h = g'x_{hg} = x_{hgg'}$.

We mostly concentrate on finitely-generated groups $G$. A good mental picture of $G$ is then its left Cayley graph with vertices $G$ and edges $(g, sg)$ for $s \in S$ where $S$ a generating set of $G$. In this picture, $x \mapsto gx$ should be thought of as pulling the node $g$ to the origin $e$. Alternatively, one may think of the Cayley graph as having a floating origin, and in $gx$, the origin has moved to $gx$. 

If $C \subset X$ is clopen in a subshift $X \subset A^G$, we say $C$ \emph{appears} or \emph{occurs} $g \in G$ if $gx \in C$, and $g$ is an \emph{occurrence} of $C$. Note that this means that the ``local pattern at $g$ in $x$'' is $C$. Specifically, a clopen set is a finite union of \emph{cylinders} $[p]$ where $p \in A^D$, $D \Subset G$ is a pattern and $[p] = \{x \in X \;|\; x|D = p\}$. A cylinder $[p]$ appears at $g$ if and only if $x_{hg} = gx_h = p_h$ for all $h \in D$, which precisely means that in the left Cayley graph, we are looking at the relative pattern around the position $g$. For a pattern $p \in A^D$, we also use directly the appearance/occurrence terminology, identifying $p$ with $[p]$.

A subshift was already defined in the introduction as a topologically closed $G$-invariant subset of the full shift $A^G$, equivalently this is a subset of $A^G$ defined by a (possibly infinite) set $\mathcal{F}$ of \emph{forbidden patterns} $p \in A^D$ for finite $D \Subset G$, meaning $X = \{x \in A^G \;|\; \forall p \in F: \forall g \in G: gx|D \neq p\}$.

\begin{definition}
\label{def:SFT}
A \emph{subshift of finite type}, or \emph{SFT}, is $X \subset A^G$ such that for some clopen set $C \subset A^G$, we have $X = \{x \in A^G \;|\; \forall g \in G: gx \in C\}$.
\end{definition}

A \emph{conjugacy} between subshifts $X \subset A^G$ and $Y \subset B^G$ is a homeomorphism $\phi : X \to Y$ such that $g \circ \phi = \phi \circ g$ (where group elements are identified by their corresponding shift maps). The automorphism group of a subshift is just the group of its self-conjugacies. A conjugacy is also called a \emph{recoding}. A typical example of a recoding is a \emph{block coding}, where we take $D \Subset G$ with $e \in D$, and define $\phi(x)_g = p \in A^D$ where $p_h = x_{hg}$. This conjugates $X$ with a subshift of $(A^D)^G$.

If $K \subset G$, we say $x \in A^G$ is \emph{$K$-periodic} if $\forall g \in K: gx = x$.

Much of the present paper studies automorphism groups of direct product subshifts $X \times Y$ where $X \subset A^G$, $Y \subset B^G$, and $G$ acts diagonally by $g(x, y) = (gx, gy)$. We can see $X \times Y$ equivalently as a subshift on the group $(A \times B)^G$ by mapping $(x, y) \in X \times Y$ to $z \in (A \times B)^G$ where $z_g = (x_g, y_g)$. This should not cause any confusion.

We use standard terminology from the theory of cellular automata in the context of direct products: we think of $(x, y) \in X \times Y$ as having two \emph{tracks}. The mental image is that the configuration $y$ is ``under'' $x$. One should imagine two left Cayley graphs of $G$ on top of one another, called the \emph{top track} (containing elements from $X$) and the \emph{bottom track} (containing elements from $Y$). In particular, we say that a clopen set $D \subset Y$ \emph{appears under an occurrence of a clopen set $C \subset X$ at $g \in G$} in a configuration $(x, y)$ if $gx \in C$ (i.e.\ $C$ appears at $g$ in $x$) and $gy \in D$ (i.e.\ $D$ appears at $g$ in $y$).

The commutator of $g, g'$ is $[g, h] = g^{-1}h^{-1}gh$. \emph{Left conjugation} of $g$ by $h$ is ${}^hg = hgh^{-1}$, and \emph{right conjugation} is $g^h = h^{-1}gh$. We write also ${}^{-h}g = {}^h(g^{-1})$ and $g^{-h} = (g^{-1})^h$. Left conjugation is a group action, i.e.\ ${}^a ({}^b c) = {}^{ab}c$. The \emph{left cosets} of a subgroup $K \leq G$ are $gK$ for $g \in G$. Two subgroups $H, K \leq G$ \emph{commute} if $hk = kh$ for all $h \in H$, $k \in K$. If $g_1, g_2, \ldots, g_n \in G$, we call the product $g_1 g_2 \cdots g_n$, a \emph{commuting product} if $g_ig_j = g_jg_i$ for all $i, j \in \{1, \ldots, n\}$. We also write this as $\prod_i g_i$.

A \emph{quotient} of a group is a group that appears as its image in some group homomorphism. The homomorphism itself is referred to as a \emph{quotient map}. A group is \emph{simple} if it has no nontrivial quotients (every quotient map goes to the trivial group, or is an isomorphism). Alternating groups on at least five elements are simple. The \emph{normal closure} of an element $g \in G$ of a group is the smallest normal subgroup containing $g$. Concretely, this is the group generated by all conjugates of $g$. Of course, in a simple group, the normal closure is always all of $G$ (unless $g$ is trivial).

The \emph{lower central series} of a group $G$ is defined inductively by $G_1 = [G, G]$, $G_{i+1} = [G, G_i]$, where $[A, B]$ denotes the group generated inside $G$ by commutators $[a, b]$ with $a \in A, b \in B$. We only need that the lower central series of the symmetric group on at least three elements terminates in the corresponding alternating group.

A subset $S \subset G$ is \emph{symmetric} if $\forall s \in S: s^{-1} \in S$. When $G$ is finitely-generated and a symmetric generating set $S \ni e_G$ is clear from context, we say two sets $C, D \subset G$ \emph{touch} if $S_G C \cap D \neq \emptyset$. This is a symmetric reflexive (but not transitive) relation.

For permutation groups we use the usual cycle notation $(a_1 \; a_2 \; \cdots \; a_k)$ to denote the permutation that takes $a_i$ to $a_{i+1}$ for $i < k$ and takes $a_k$ to $a_1$. When the $a_i$ are longer expressions, we use the separator $;$ instead of a space, e.g.\ $(00;\; 01;\; 10;\; 11)$ is a $4$-cycle of words of length $2$. 

If $G$ is a finitely-generated group, we say $B \subset G$ is \emph{right-syndetic} if $SB = G$ for some finite set $S \Subset G$. Note that this means precisely that $B$ is relatively dense in the left Cayley graph, i.e.\ every $g \in G$ is at bounded distance from an element of $B$. The \emph{$S$-interior} of $B \subset G$ is $\{g \in B \;|\; Sg \subset B\}$. If $x \in A^G$, we say $p \in A^D$ \emph{appears syndetically} if the set of positions $K \subset G$ where $p$ occurs is (right) syndetic.

\section{The process}

In this section we describe a process that produces, from a configuration $x \in A^G$, a sequence of increasingly coarsened partitions of a group $G$ together with a distinguished representative in each partition element. This is in spirit a variant of the usual marker lemma in symbolic dynamics. While there are versions of the marker lemma for groups (e.g.\ \cite{Me23}), our interest is not so much in separation of markers, but instead we need to leverage the cardinality of orbits. We give a direct proof of precisely the statement we need, which seems to be somewhat orthogonal to that of \cite{Me23}.

The process can be described as follows: We think of the Cayley graph of the group $G$, with each element $g \in G$ containing the symbol $x_g$ from some point $x \in X$ for a subshift $X$. Initially, each $g \in G$ is considered a singleton ``cell colony'' (inspired by the fact that nodes $g$ are usually called cells in cellular automata theory). According to a set of essentially arbitrary, but deterministic and spatially homogeneous, rules, the cell colonies begin to merge with nearby colonies to form larger colonies. Each colony furthermore has a single brain cell, which commands its merging (and the brain cells play a key role in the application to automorphism group).


The idea is that we take the point of view of the brain cells, and describe their behavior (and thus the behavior of the colonies they command) only based on the relative configuration around them, and the relative configuration of other colonies nearby. As long as the description of the behavior does not use the actual position $g \in G$ of the brain cell, and each brain cell follows the same procedure, we will necessarily obtain a shift-invariant set of colonies in the limit. We only need to ensure that we cannot get stuck forever, and should ensure that the brain cells try everything in their power to break all possible translational symmetries.

We start with a standard observation.

\begin{lemma}
Let $G$ be a finitely-generated group and $X \subset A^G$ be a subshift. Then for all $n$, there exists a clopen set $C_n \subset X$ such that $x \in X$ has orbit of size at most $n$ if and only if $gx \in C_n$ for all $g \in G$.
\end{lemma}

\begin{proof}
It suffices to show this for $X = A^G$ (if $C$ is the clopen set obtained for $A^G$, for $X$ we can use $C \cap X$). The points in $A^G$ whose orbit has size at most $n$ have point stabilizer of index at most $n$. There are finitely many such subgroups. Thus, the set of points whose orbit has size at most $n$ is a finite subshift of $X$. It is well-known that a finite subshift on a finitely-generated group is of finite type, equivalently it is the set of points whose orbit stays in a particular clopen set $C$, and we can take $C_n = C$.
\end{proof}

Formally, we now consider partitions of (or equivalence relations on) a set $G$ with a choice of representatives in the partition elements, and with finite partition elements. To correspond to the mental image we have, we will refer to equivalence classes as \emph{cell colonies} (or just \emph{colonies}), and to the representatives as \emph{brain cell} (or just \emph{brains}) in each of them, and at discrete time steps some of the brains will yield control of their colony to another brain nearby (i.e.\ at a bounded distance, or number of edges away in the left Cayley graph). For short, we refer to a partition of $G$ with a choice of representatives as a \emph{culture} on $G$.

We now put a topology on the set of cultures of $G$, when $G$ is a finitely-generated group. For this, we code a culture in $(2 \times \mathcal{P}_{\mathrm{fin}}(G))^G$, by using the bit $b \in 2$ at $g$ to store whether the current node is a brain, and $p \in \mathcal{P}(G)$ encodes the relative shape of colony the current node $g$ belongs to, meaning $g$ is in relation precisely with the nodes $pg$. On $\mathcal{P}_{\mathrm{fin}}(G)$ we use the discrete topology and on a set of the form $A^G$ we use the product topology, so the encodings live in the Baire space. We will in practice only work with cultures where only finitely many sets in $\mathcal{P}_{\mathrm{fin}}(G)$ are used, equivalently the colonies are bounded in the left Cayley graph of $G$. This means that we are really working with a countable union of Cantor spaces
\[ \bigcup_{P \Subset \mathcal{P}_{\mathrm{fin}}} (2 \times P)^G \]
with the final topology.

Let us axiomatize the set of points of $(2 \times \mathcal{P}_{\mathrm{fin}}(G))^G$ corresponding to (codings of) cultures: if $(x, y) \in 2^G \times \mathcal{P}_{\mathrm{fin}}(G)^G$, we want 
\begin{enumerate}
\item $|\{y_g \;|\; g \in G\}| < \infty$,
\item $\forall g \in G: e \in y_G$,
\item $\forall h, g: h \in y_g \implies h^{-1} \in y_{hg}$,
\item $\forall h, g: h \in y_g \wedge k \in y_{hg} \implies kh \in y_g$,
\item $\forall g: \exists h: h \in y_g \wedge x_{hg} = 1$.
\item $\forall h, g: h \in y_g \wedge x_g = 1 \implies h = e \vee x_{hg} = 0$.
\end{enumerate}

There is an action of $G$ on cultures: we simply shift the configuration $(x, y)$ coding a culture according to the usual shift convention, to obtain a new coding $(gx, gy)$ of a culture. Directly in terms of cultures, if $z$ is a culture where $S$ is a colony with brain $h$, then $gz$ is a colony where $Sg^{-1}$ is a colony with brain $hg^{-1}$.

\textbf{The process.} Given $x \in A^G$, we now describe a process, which we that describes a sequence of cultures on $x$ evolving in time, such that -- at least hopefully -- all colonies will eventually grow (by absorbing other colonies or merging into other colonies). We first describe \emph{the basic process}, which can be modified to obtain additional properties.

A key technical notion here is \emph{merging} of colonies. Suppose we have specified a (possibly infinite) set of colonies $(C_1^i, C_2^i)$ in a culture $z$. Then we can form a new equivalence relation, by adding to the existing equivalence relation all pairs $(c_1, c_2)$ with $c_1 \in C_1^i$ and $c_2 \in C_2^i$, and taking the transitive closure. This means that every partition element of the form $C_1^i$ or $C_2^j$ disappears, and such sets become part of new larger colonies. If every new colony contains only one colony of the form $C_2^j$, we say the merge is \emph{good}. In this case, we obtain a new culture, where the brain of every new colony is the brain of the $C_2^j$ colony it contains.

We think of $C_2^i$-colonies as \emph{absorbing} the colonies $C_1^i$. In turn we say $C_1^i$ \emph{fuses} into $C_2^i$. We use \emph{merge} as a neutral term for this action. We also apply the terminologies to cells $g \in G$, understanding that the terms then refer to the colonies containing the $g$.

\begin{lemma}
\label{lem:Merging}
Suppose $z$ is a culture, and $(C_1^i, C_2^i)_{i \in \N}$ is a sequence of pairs of colonies such that $C_1^i$ touches $C_2^j$. Suppose that
\begin{itemize}
\item $C_1^i \neq C_2^j$ for all $i, j$, and
\item $C_1^i \neq C_1^j$ whenever $i \neq j$.
\end{itemize}
Then the corresponding merge is good.
\end{lemma}

The conditions can equivalently be stated as follows: colonies $C_1^j$ that fuse into others do not themselves absorb any colonies, no pair is listed twice, and no colony is absorbed into two colonies.

\begin{proof}
Consider the graph where nodes are colonies that appear in the pairs. By the first item, this graph is bipartite, with ``left nodes'' thes ones that appear as $C_1^i$ and ``right nodes'' the ones that appear as $C_2^j$. Directing the edges left-to-right, by the second item, the in-degree is $1$, which makes it clear that the new colonies are precisely the sets of the form
\[ C_2^j \cup \bigcup \{C_1^i \;|\; i \in \N, C_1^j = C_2^j\}, \]
thus they all have exactly one absorbing $C_2^j$-cell. We also see that each $C_1^i$ fused into $C_2^j$ touches $C_2^j$, which gives a bound on the diameter of the new colonies, so the colonies are bounded, as required in the definition of a culture.
\end{proof}

We start with the discrete culture, where each node $g \in G$ is the brain of its own singleton colony $\{g\}$. First, we may perform some shift-commuting preprocessing steps, where some of the colonies merge.

A common one is to perform an \emph{$S$-merge at each occurrence of a clopen set $C$}. To define this, we define an auxiliary notion that will also be of importance later.

\begin{definition}
\label{def:Safety}
Let $X \subset A^G$ be a subshift and $C \subset X$ clopen. We say $S \subset G$ is \emph{safe} for $C$ if $s^{-1}s' C \cap C = \emptyset$ for any $s, s' \in S$ with $s \neq s'$.
\end{definition}

We will only perform $S$-merges at every occurrence of $C$ when $e \in S \Subset G$ and $S$ is safe for $C$, and the definition of such a step is that those $g$ satisfying $gx \in C$, have their colonies fused with those of all $sg$. Note that such merges are shift-invariant (meaning the result of applying the merges to $x$ is the same as applying them to $gx$, and then shifting the resulting culture by $g^{-1}$).

The point of the safety property here is that there is no ``unintended'' additional merging of colonies (i.e.\ the assumptions of Lemma~\ref{lem:Merging} hold): If $h$ is fused with the colonies of two different brains $g, g'$ in this step, then $h = sg = s'g'$ for $s, s' \in S$ implying both $g'x$ and $gx = s^{-1}s'g'x$ are in $C$, meaning $s^{-1}sC \cap C \neq \emptyset$. This cannot happen if $S$ is safe for $C$.

After the possible preprocessing steps, the main part of the process begins on the resulting culture. There is a lot of freedom in the process, specifically an infinite amount of choices of how to order the various possible steps (and the main conclusion Lemma~\ref{lem:Stable} is valid in any case), but we make it deterministic artificially by fixing an ordering, and this is important in Lemma~\ref{lem:Eventually}.

The process always runs forever, in infinitely many steps, and we describe these steps in what follows. First, pick any total order on $A^{B_r}$ for all $r$, for example a lexicographic order for some total orders on $A$ and $B_r$. We refer to these as the \emph{pattern orders} (one for each $r$). No consistency whatsoever is needed for these orders. 

Then, we let $(k, h, p)$ range over all quadruples where $k \in \N$, $h \in G$, and $p \in A^{B_r}$ for some $r \in \N$. For each such triple in some order, we perform a \emph{step}, which in turn consists of the following two \emph{substeps}:
\begin{enumerate}
\item If $g$ is a brain cell whose colony is of size precisely $k$, $hg$ is contained in a strictly larger colony, and the two colonies touch, then the two colonies merge, with the brain of the colony of $hg$ becoming the brain of the new colony.
\item If $g$ and $hg$ are brains, $x_{kg} = p_k$ for all $k \in \mathcal{B}_r$, and $k \mapsto x_{khg}$ in turn is a pattern higher in pattern order of $A^{B_r}$, and the two colonies touch, then the colonies merge, with $hg$ becoming the brain of the new colony.
\end{enumerate}

Note that in the first item, $g$ needs to be a brain, while $hg$ need not. In the second, we require that both are brains. In both cases, the colony of $hg$ absorbs that of $g$. Merges described in the first item are called \emph{cardinality merges}, and those in the second item are called \emph{pattern merges}.

We say the same in formulas: in a cardinality merge, we apply Lemma~\ref{lem:Merging} with pairs $(C_1^i, C_2^i)$ such that, if $g$ is the brain of $C_1^i$, we have
\[ hg \in C_2^i \wedge |C_1^i| = k \wedge |C_2^i| > k, \]
and in a pattern merge, we apply Lemma~\ref{lem:Merging} with pairs $(C_1^i, C_2^i)$ such that, letting $g$ be the brain of $C_1^i$, $hg$ is the brain of $C_2^i$, and
\[  hg x|\mathcal{B}_r < gx|\mathcal{B}_r = p. \]

It is easy to see that the assumptions of Lemma~\ref{lem:Merging} hold in both cases. In the case of a cardinality merge, each $C_2^j$ is larger than $k$ and each $C_1^j$ has size $k$, so the first assumption of the lemma holds. The second assumption holds because for a colony of the form $C_1^i$, $C_2^i$ is precisely the colony of $hg$, where $g$ is the brain of $C_1^i$. For pattern merges, the argument is similar but easier.

We conclude that each step (consisting of the two substeps) indeed produces a new culture. This concludes the description of the basic process.

We consider mainly two variants of basic the process. The \emph{cardinality process} works exactly as the process, but we skip all the pattern merges. We still range over triples $(k, h, p)$, so that each cardinality $k$ and direction $h$ is considered infinitely many times. 


It is easy to see that the function that produces the culture after the $n$th step of any of the processes is continuous and shift-commuting -- it is defined by a shift-invariant process that can be followed at each $g \in G$ separately, with all elements agreeing on the outcomes. Although we always continue the process infinitely, it may happen that after some step, no further mergings happen. We then say we have reached a \emph{stable culture}.

\begin{example}
We show an example of a possible result from running the basic process on the group $\Z^2$, starting with a periodic configuration. We consider a $\Z^2$-configuration with periods $(6, 0)$ and $(0, 5)$ with the initial culture of all cells forming singleton colonies. We draw the left Cayley graph (which is also the right Cayley graph) of $\Z^2$ in standard orientation and with the standard generators, with $(1,0)$ pointing to the right and $(0,1)$ upward. We mark the brains with blue dots, their colonies with thick borders, and we use shades of gray to mark $0,1,2$ ($0$ is white, $2$ is dark gray). An $18$-by-$8$ area is shown
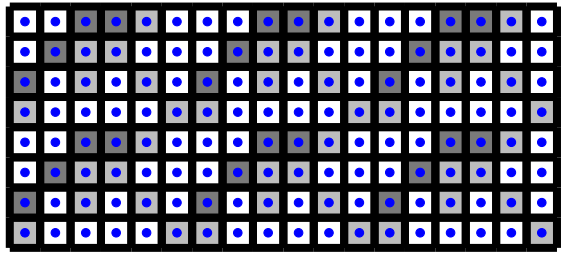
\begin{center}
\input{config1}
\end{center}
We first perform a pattern merge with pattern $(0,0) \mapsto 0$ in direction $(0, 1)$. Thus, each brain $g$ checks if $(0,1)+g$ is a brain with a lexicographically larger pattern around it (i.e.\, contains symbol $1$ or $2$), and if so, $g$ fuses into $(0,1) + g$:
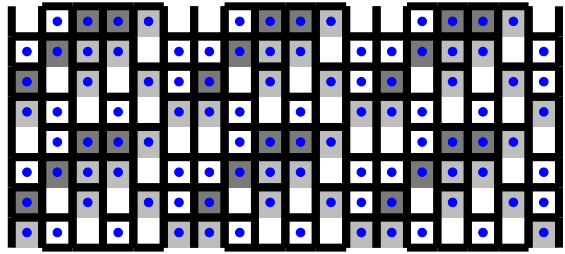
\begin{center}
\input{config2}
\end{center}
Next, we perform a cardinality merge in direction $(1, 1)$ with cardinality $1$. Thus, each brain $g$ with colony of size 1 fuses into the colony of of $(1,1) + g$ if the colony containing $g + (1,1)$ is strictly larger than the colony of $g$:
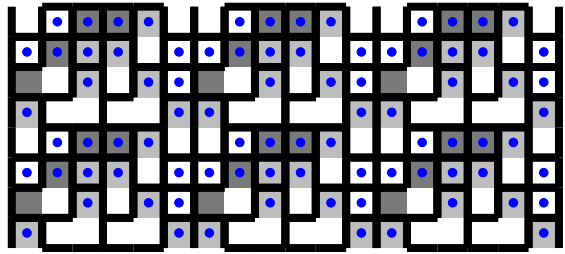
\begin{center}
\input{config3}
\end{center}
Applying a random sequence of cardinality merges, we reach the following stable culture (of course, many results are possible depending on the choices of orders in the definition of the process):
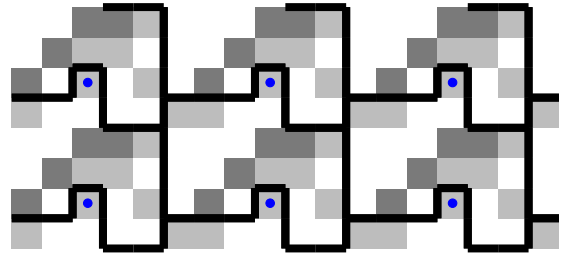
\begin{center}
\begin{tikzpicture}[scale=0.4]
\fill[black!26!white] (0,0) rectangle (1,1);
\fill[black!52!white] (0,1) rectangle (1,2);
\fill[white] (0,2) rectangle (1,3);
\fill[white] (0,3) rectangle (1,4);
\fill[black!26!white] (0,4) rectangle (1,5);
\fill[black!52!white] (0,5) rectangle (1,6);
\fill[white] (0,6) rectangle (1,7);
\fill[white] (0,7) rectangle (1,8);
\fill[white] (1,0) rectangle (2,1);
\fill[white] (1,1) rectangle (2,2);
\fill[black!52!white] (1,2) rectangle (2,3);
\fill[white] (1,3) rectangle (2,4);
\fill[white] (1,4) rectangle (2,5);
\fill[white] (1,5) rectangle (2,6);
\fill[black!52!white] (1,6) rectangle (2,7);
\fill[white] (1,7) rectangle (2,8);
\fill[white] (2,0) rectangle (3,1);
\fill[black!26!white] (2,1) rectangle (3,2);
\draw[fill,blue] (2.5,1.5) circle (0.14);
\fill[black!26!white] (2,2) rectangle (3,3);
\fill[black!52!white] (2,3) rectangle (3,4);
\fill[white] (2,4) rectangle (3,5);
\fill[black!26!white] (2,5) rectangle (3,6);
\draw[fill,blue] (2.5,5.5) circle (0.14);
\fill[black!26!white] (2,6) rectangle (3,7);
\fill[black!52!white] (2,7) rectangle (3,8);
\fill[white] (3,0) rectangle (4,1);
\fill[white] (3,1) rectangle (4,2);
\fill[black!26!white] (3,2) rectangle (4,3);
\fill[black!52!white] (3,3) rectangle (4,4);
\fill[white] (3,4) rectangle (4,5);
\fill[white] (3,5) rectangle (4,6);
\fill[black!26!white] (3,6) rectangle (4,7);
\fill[black!52!white] (3,7) rectangle (4,8);
\fill[white] (4,0) rectangle (5,1);
\fill[black!26!white] (4,1) rectangle (5,2);
\fill[white] (4,2) rectangle (5,3);
\fill[black!26!white] (4,3) rectangle (5,4);
\fill[white] (4,4) rectangle (5,5);
\fill[black!26!white] (4,5) rectangle (5,6);
\fill[white] (4,6) rectangle (5,7);
\fill[black!26!white] (4,7) rectangle (5,8);
\fill[black!26!white] (5,0) rectangle (6,1);
\fill[white] (5,1) rectangle (6,2);
\fill[white] (5,2) rectangle (6,3);
\fill[white] (5,3) rectangle (6,4);
\fill[black!26!white] (5,4) rectangle (6,5);
\fill[white] (5,5) rectangle (6,6);
\fill[white] (5,6) rectangle (6,7);
\fill[white] (5,7) rectangle (6,8);
\fill[black!26!white] (6,0) rectangle (7,1);
\fill[black!52!white] (6,1) rectangle (7,2);
\fill[white] (6,2) rectangle (7,3);
\fill[white] (6,3) rectangle (7,4);
\fill[black!26!white] (6,4) rectangle (7,5);
\fill[black!52!white] (6,5) rectangle (7,6);
\fill[white] (6,6) rectangle (7,7);
\fill[white] (6,7) rectangle (7,8);
\fill[white] (7,0) rectangle (8,1);
\fill[white] (7,1) rectangle (8,2);
\fill[black!52!white] (7,2) rectangle (8,3);
\fill[white] (7,3) rectangle (8,4);
\fill[white] (7,4) rectangle (8,5);
\fill[white] (7,5) rectangle (8,6);
\fill[black!52!white] (7,6) rectangle (8,7);
\fill[white] (7,7) rectangle (8,8);
\fill[white] (8,0) rectangle (9,1);
\fill[black!26!white] (8,1) rectangle (9,2);
\draw[fill,blue] (8.5,1.5) circle (0.14);
\fill[black!26!white] (8,2) rectangle (9,3);
\fill[black!52!white] (8,3) rectangle (9,4);
\fill[white] (8,4) rectangle (9,5);
\fill[black!26!white] (8,5) rectangle (9,6);
\draw[fill,blue] (8.5,5.5) circle (0.14);
\fill[black!26!white] (8,6) rectangle (9,7);
\fill[black!52!white] (8,7) rectangle (9,8);
\fill[white] (9,0) rectangle (10,1);
\fill[white] (9,1) rectangle (10,2);
\fill[black!26!white] (9,2) rectangle (10,3);
\fill[black!52!white] (9,3) rectangle (10,4);
\fill[white] (9,4) rectangle (10,5);
\fill[white] (9,5) rectangle (10,6);
\fill[black!26!white] (9,6) rectangle (10,7);
\fill[black!52!white] (9,7) rectangle (10,8);
\fill[white] (10,0) rectangle (11,1);
\fill[black!26!white] (10,1) rectangle (11,2);
\fill[white] (10,2) rectangle (11,3);
\fill[black!26!white] (10,3) rectangle (11,4);
\fill[white] (10,4) rectangle (11,5);
\fill[black!26!white] (10,5) rectangle (11,6);
\fill[white] (10,6) rectangle (11,7);
\fill[black!26!white] (10,7) rectangle (11,8);
\fill[black!26!white] (11,0) rectangle (12,1);
\fill[white] (11,1) rectangle (12,2);
\fill[white] (11,2) rectangle (12,3);
\fill[white] (11,3) rectangle (12,4);
\fill[black!26!white] (11,4) rectangle (12,5);
\fill[white] (11,5) rectangle (12,6);
\fill[white] (11,6) rectangle (12,7);
\fill[white] (11,7) rectangle (12,8);
\fill[black!26!white] (12,0) rectangle (13,1);
\fill[black!52!white] (12,1) rectangle (13,2);
\fill[white] (12,2) rectangle (13,3);
\fill[white] (12,3) rectangle (13,4);
\fill[black!26!white] (12,4) rectangle (13,5);
\fill[black!52!white] (12,5) rectangle (13,6);
\fill[white] (12,6) rectangle (13,7);
\fill[white] (12,7) rectangle (13,8);
\fill[white] (13,0) rectangle (14,1);
\fill[white] (13,1) rectangle (14,2);
\fill[black!52!white] (13,2) rectangle (14,3);
\fill[white] (13,3) rectangle (14,4);
\fill[white] (13,4) rectangle (14,5);
\fill[white] (13,5) rectangle (14,6);
\fill[black!52!white] (13,6) rectangle (14,7);
\fill[white] (13,7) rectangle (14,8);
\fill[white] (14,0) rectangle (15,1);
\fill[black!26!white] (14,1) rectangle (15,2);
\draw[fill,blue] (14.5,1.5) circle (0.14);
\fill[black!26!white] (14,2) rectangle (15,3);
\fill[black!52!white] (14,3) rectangle (15,4);
\fill[white] (14,4) rectangle (15,5);
\fill[black!26!white] (14,5) rectangle (15,6);
\draw[fill,blue] (14.5,5.5) circle (0.14);
\fill[black!26!white] (14,6) rectangle (15,7);
\fill[black!52!white] (14,7) rectangle (15,8);
\fill[white] (15,0) rectangle (16,1);
\fill[white] (15,1) rectangle (16,2);
\fill[black!26!white] (15,2) rectangle (16,3);
\fill[black!52!white] (15,3) rectangle (16,4);
\fill[white] (15,4) rectangle (16,5);
\fill[white] (15,5) rectangle (16,6);
\fill[black!26!white] (15,6) rectangle (16,7);
\fill[black!52!white] (15,7) rectangle (16,8);
\fill[white] (16,0) rectangle (17,1);
\fill[black!26!white] (16,1) rectangle (17,2);
\fill[white] (16,2) rectangle (17,3);
\fill[black!26!white] (16,3) rectangle (17,4);
\fill[white] (16,4) rectangle (17,5);
\fill[black!26!white] (16,5) rectangle (17,6);
\fill[white] (16,6) rectangle (17,7);
\fill[black!26!white] (16,7) rectangle (17,8);
\fill[black!26!white] (17,0) rectangle (18,1);
\fill[white] (17,1) rectangle (18,2);
\fill[white] (17,2) rectangle (18,3);
\fill[white] (17,3) rectangle (18,4);
\fill[black!26!white] (17,4) rectangle (18,5);
\fill[white] (17,5) rectangle (18,6);
\fill[white] (17,6) rectangle (18,7);
\fill[white] (17,7) rectangle (18,8);
\draw[line width=3] (0,1) -- (1,1);
\draw[line width=3] (0,1) -- (1,1);
\draw[line width=3] (0,5) -- (1,5);
\draw[line width=3] (0,5) -- (1,5);
\draw[line width=3] (1,1) -- (2,1);
\draw[line width=3] (2,1) -- (2,2);
\draw[line width=3] (1,1) -- (2,1);
\draw[line width=3] (1,5) -- (2,5);
\draw[line width=3] (2,5) -- (2,6);
\draw[line width=3] (1,5) -- (2,5);
\draw[line width=3] (3,0) -- (3,1);
\draw[line width=3] (2,1) -- (2,2);
\draw[line width=3] (3,1) -- (3,2);
\draw[line width=3] (2,2) -- (3,2);
\draw[line width=3] (2,2) -- (3,2);
\draw[line width=3] (3,4) -- (3,5);
\draw[line width=3] (2,5) -- (2,6);
\draw[line width=3] (3,5) -- (3,6);
\draw[line width=3] (2,6) -- (3,6);
\draw[line width=3] (2,6) -- (3,6);
\draw[line width=3] (3,0) -- (3,1);
\draw[line width=3] (3,0) -- (4,0);
\draw[line width=3] (3,1) -- (3,2);
\draw[line width=3] (3,4) -- (4,4);
\draw[line width=3] (3,4) -- (3,5);
\draw[line width=3] (3,4) -- (4,4);
\draw[line width=3] (3,5) -- (3,6);
\draw[line width=3] (3,8) -- (4,8);
\draw[line width=3] (5,0) -- (5,1);
\draw[line width=3] (4,0) -- (5,0);
\draw[line width=3] (5,1) -- (5,2);
\draw[line width=3] (5,2) -- (5,3);
\draw[line width=3] (5,3) -- (5,4);
\draw[line width=3] (4,4) -- (5,4);
\draw[line width=3] (5,4) -- (5,5);
\draw[line width=3] (4,4) -- (5,4);
\draw[line width=3] (5,5) -- (5,6);
\draw[line width=3] (5,6) -- (5,7);
\draw[line width=3] (5,7) -- (5,8);
\draw[line width=3] (4,8) -- (5,8);
\draw[line width=3] (5,0) -- (5,1);
\draw[line width=3] (5,1) -- (6,1);
\draw[line width=3] (5,1) -- (5,2);
\draw[line width=3] (5,1) -- (6,1);
\draw[line width=3] (5,2) -- (5,3);
\draw[line width=3] (5,3) -- (5,4);
\draw[line width=3] (5,4) -- (5,5);
\draw[line width=3] (5,5) -- (6,5);
\draw[line width=3] (5,5) -- (5,6);
\draw[line width=3] (5,5) -- (6,5);
\draw[line width=3] (5,6) -- (5,7);
\draw[line width=3] (5,7) -- (5,8);
\draw[line width=3] (6,1) -- (7,1);
\draw[line width=3] (6,1) -- (7,1);
\draw[line width=3] (6,5) -- (7,5);
\draw[line width=3] (6,5) -- (7,5);
\draw[line width=3] (7,1) -- (8,1);
\draw[line width=3] (8,1) -- (8,2);
\draw[line width=3] (7,1) -- (8,1);
\draw[line width=3] (7,5) -- (8,5);
\draw[line width=3] (8,5) -- (8,6);
\draw[line width=3] (7,5) -- (8,5);
\draw[line width=3] (9,0) -- (9,1);
\draw[line width=3] (8,1) -- (8,2);
\draw[line width=3] (9,1) -- (9,2);
\draw[line width=3] (8,2) -- (9,2);
\draw[line width=3] (8,2) -- (9,2);
\draw[line width=3] (9,4) -- (9,5);
\draw[line width=3] (8,5) -- (8,6);
\draw[line width=3] (9,5) -- (9,6);
\draw[line width=3] (8,6) -- (9,6);
\draw[line width=3] (8,6) -- (9,6);
\draw[line width=3] (9,0) -- (9,1);
\draw[line width=3] (9,0) -- (10,0);
\draw[line width=3] (9,1) -- (9,2);
\draw[line width=3] (9,4) -- (10,4);
\draw[line width=3] (9,4) -- (9,5);
\draw[line width=3] (9,4) -- (10,4);
\draw[line width=3] (9,5) -- (9,6);
\draw[line width=3] (9,8) -- (10,8);
\draw[line width=3] (11,0) -- (11,1);
\draw[line width=3] (10,0) -- (11,0);
\draw[line width=3] (11,1) -- (11,2);
\draw[line width=3] (11,2) -- (11,3);
\draw[line width=3] (11,3) -- (11,4);
\draw[line width=3] (10,4) -- (11,4);
\draw[line width=3] (11,4) -- (11,5);
\draw[line width=3] (10,4) -- (11,4);
\draw[line width=3] (11,5) -- (11,6);
\draw[line width=3] (11,6) -- (11,7);
\draw[line width=3] (11,7) -- (11,8);
\draw[line width=3] (10,8) -- (11,8);
\draw[line width=3] (11,0) -- (11,1);
\draw[line width=3] (11,1) -- (12,1);
\draw[line width=3] (11,1) -- (11,2);
\draw[line width=3] (11,1) -- (12,1);
\draw[line width=3] (11,2) -- (11,3);
\draw[line width=3] (11,3) -- (11,4);
\draw[line width=3] (11,4) -- (11,5);
\draw[line width=3] (11,5) -- (12,5);
\draw[line width=3] (11,5) -- (11,6);
\draw[line width=3] (11,5) -- (12,5);
\draw[line width=3] (11,6) -- (11,7);
\draw[line width=3] (11,7) -- (11,8);
\draw[line width=3] (12,1) -- (13,1);
\draw[line width=3] (12,1) -- (13,1);
\draw[line width=3] (12,5) -- (13,5);
\draw[line width=3] (12,5) -- (13,5);
\draw[line width=3] (13,1) -- (14,1);
\draw[line width=3] (14,1) -- (14,2);
\draw[line width=3] (13,1) -- (14,1);
\draw[line width=3] (13,5) -- (14,5);
\draw[line width=3] (14,5) -- (14,6);
\draw[line width=3] (13,5) -- (14,5);
\draw[line width=3] (15,0) -- (15,1);
\draw[line width=3] (14,1) -- (14,2);
\draw[line width=3] (15,1) -- (15,2);
\draw[line width=3] (14,2) -- (15,2);
\draw[line width=3] (14,2) -- (15,2);
\draw[line width=3] (15,4) -- (15,5);
\draw[line width=3] (14,5) -- (14,6);
\draw[line width=3] (15,5) -- (15,6);
\draw[line width=3] (14,6) -- (15,6);
\draw[line width=3] (14,6) -- (15,6);
\draw[line width=3] (15,0) -- (15,1);
\draw[line width=3] (15,0) -- (16,0);
\draw[line width=3] (15,1) -- (15,2);
\draw[line width=3] (15,4) -- (16,4);
\draw[line width=3] (15,4) -- (15,5);
\draw[line width=3] (15,4) -- (16,4);
\draw[line width=3] (15,5) -- (15,6);
\draw[line width=3] (15,8) -- (16,8);
\draw[line width=3] (17,0) -- (17,1);
\draw[line width=3] (16,0) -- (17,0);
\draw[line width=3] (17,1) -- (17,2);
\draw[line width=3] (17,2) -- (17,3);
\draw[line width=3] (17,3) -- (17,4);
\draw[line width=3] (16,4) -- (17,4);
\draw[line width=3] (17,4) -- (17,5);
\draw[line width=3] (16,4) -- (17,4);
\draw[line width=3] (17,5) -- (17,6);
\draw[line width=3] (17,6) -- (17,7);
\draw[line width=3] (17,7) -- (17,8);
\draw[line width=3] (16,8) -- (17,8);
\draw[line width=3] (17,0) -- (17,1);
\draw[line width=3] (17,1) -- (18,1);
\draw[line width=3] (17,1) -- (17,2);
\draw[line width=3] (17,1) -- (18,1);
\draw[line width=3] (17,2) -- (17,3);
\draw[line width=3] (17,3) -- (17,4);
\draw[line width=3] (17,4) -- (17,5);
\draw[line width=3] (17,5) -- (18,5);
\draw[line width=3] (17,5) -- (17,6);
\draw[line width=3] (17,5) -- (18,5);
\draw[line width=3] (17,6) -- (17,7);
\draw[line width=3] (17,7) -- (17,8);
\end{tikzpicture}

\end{center}

Note that this can be seen as an instance of the basic process, where pattern merges simply did not match anything. We make some remarks about the resulting culture. As expected, we see precisely $6$ brains, since the period repeats $6$ times. Indeed, we reached a connected set of coset representatives (i.e.\ connected fundamental domain) for the period lattice $6\Z \times 5\Z$. Note that the brains happen to have symbol $1$, despite $1 < 2$. Since we only performed one pattern merge, actually the process did not (yet) distinguish between the symbols $1$ and $2$. If the pattern had more translational symmetries after projecting $1$ and $2$ to a single symbol, we would never reach a set of coset representatives without additional pattern merges. \qee
\end{example}

Let us now unfix the fixed configuration $x$, and begin to analyze the global properties of the process.

\begin{lemma}
\label{lem:Stable}
If we follow the process starting with $x \in A^G$ (possibly with some shift-commuting preprocessing steps), and the size of the colony of some $g \in G$ does not tend to infinity, then
\begin{enumerate}
\item mergings only happen on finitely many steps, and we end up with a stable culture $(y, z) \in (2 \times \mathcal{P}_{\mathrm{fin}}(G))^G$;
\item there is a subgroup $K$ of finite index $n$ such that if $y_g = 1$, then $gy_h = 1$ precisely when $h \in K$;
\item there is a connected set of coset representatives $R$ for $K$, such that when $gy_h = 1$, we have $gz_h = R$;
\item $x$ is $K$-periodic.
\end{enumerate}
\end{lemma}

\begin{proof}
Suppose that the size of the colony containing $g' \in G$ does not tend to infinity. In particular, this means that some brain $g'$ never fuses into another colony, in particular there is a brain (rather than just a cell) at some $g \in G$ whose colony is finite, say of size $n$, in the limit.

We first observe that no other colony may ever grow to a size higher than $n$. For then its size would be greater than $n$ after some finite step, and eventually we would consider a triple $(k, h, p)$ where $hg$ belongs to this larger colony, and it would want to merge with $g$. Furthermore, there must exist such a pair such that the colonies touch, for some $g, h$ (follow the path from $g$ to $hg$ to find such a pair). So eventually, $g$ would also fuse into a larger class. The same logic shows that in fact all nodes end up in colonies of size precisely $n$ in the limit.

Now translate the point so that the identity $e_G$ contains a brain. Suppose that there is another brain at $g \in G$ in the limit. Then in fact $x = gx$: if this is not the case, then $x|\mathcal{B}_r \neq gx|\mathcal{B}_r$ for some $r$, and thus one of these patterns would be larger than the other in the total order on patterns, and there must exist such a pair where the colonies touch (by following any path from $e_G$ to $g$, and considering neighboring colonies when the colony changes along the path). Eventually we will consider triples $(k, g, p)$ where $p = x|B_r$, and $(k, g^{-1}, p)$ where $p = gx|B_r$, for some $k \in \N$, and the colonies would merge.

If $K$ is the point stabilizer of $x$, then we conclude that brains can only appear on $K$ in the limit. But of course, every $g \in K$ indeed contains a brain in the limit, because the process is shift-invariant. Furthermore, again by shift-commutation of the process, these brain are already formed and reach their final colony at the step when the brain at $e$ reaches its final colony. We conclude that no mergings can happen after this finite step.

At this step, like on every step, every $g \in G$ is in some colony. If $R$ is the colony of the identity, then by shift-invariance, the colony of $g \in K$ is precisely $Rg$. Furthermore, these colonies are disjoint for distinct $g \in K$. This means precisely that $R$ is a set of left representatives for the subgroup $K$. In particular, $K$ has index $n$.

Note that since at every point in time, the colonies stay connected (because we only merge touching sets), $R$ is a connected set of left coset representatives, as claimed.
\end{proof}

When $G$ is a free group, a connected set of coset representatives is called a \emph{Schreier traversal}.

\begin{remark}
One might want a further step in the process, where the brains move inside their colony according to some logic. For example, if $x \in A^{F_2}$ is a configuration whose point stabilizer is precisely $\langle a \rangle$ for one of the free generators $a$ of $F_2$, the process might tend to a culture with brains at $\langle a \rangle$; or brains (and their pairwise distances) might tend to infinity. Both cases are possible in the present version of the process, and we might want control over which of these cases is preferred. It would also be of interest to clarify the relation with the usual marker lemmas, in particular that of \cite{Me25}. However, these issues are beyond the scope of the present paper, and in particular moving brains by any logic would add complications in the previous lemma.
\end{remark}

Of course, if brains controlling large areas pop up eventually, then in a compact set, they pop up after a bounded number of steps. This is expressed formally by the following lemmas. 
If $S \Subset G$, the \emph{$S$-count} of a set $T \subset G$ is the number of $g \in G$ such that $Sg \subset T$. (Possibly $\infty$.)

\begin{lemma}
Let $G$ be a finitely-generated group, suppose $e \in S \Subset G$, and that $X \subset A^G$ a subshift. 
Let $C$ be a clopen set, and suppose $S$ is safe for it. Then there is a preprocessing step for the process such that applying it to any $x \in X$, we reach a culture where all colonies have a positive $S$-count, and for every occurrence $g$ of $C$ the colony of $g$ contains $Sg$.
\end{lemma}

\begin{proof}
If $|S| = 1$, then $S = \{e\}$, and we can use a trivial preprocessing step where nothing happens -- the discrete culture already satisfies the conclusion.

Otherwise, the preprocessing is as follows: We immediately perform $S$-merges at each occurrence of $C$. This is possible since we assume $S$ is safe for $C$. Then no matter what happens after, for every occurrence $g$ of $C$ the colony of $g$ contains $Sg$. So it suffices to show that we can also ensure that all colonies have a positive $S$-count.

Next, since $X$ is $S^{-1} S$-aperiodic, we may recode $C$ so that $x_g \neq x_{sg}$ whenever $s \in S^{-1}S \setminus \{e\}$. If $g$ has a maximal symbol, and no merges were yet performed nearby (i.e.\ $C$ does not appear nearby), then we absorb the (singleton colonies of all) cells in $Sg$ (noting that this does not trigger any further merges coming from transitivity, again by $S^{-1}S$-aperiodicity of $X$).

We then proceed with symbols in descending order, and if the symbol at $g$ is maximal in a sufficiently large local neighborhood, it absorbs the cells in $Sg$. Specifically, we want that this does not trigger any further merges, and this happens as long as we only merge at symbols where none of the nodes $Sg$ have merged previously. We observe that at this point, every colony that is not a singleton has $S$-count $1$, and such colonies are right-syndetic in $G$. For the latter claim, if $g$ is not merged at this point, then this is because there is a larger symbol nearby that did merge.

After this, we begin the cardinality process. We claim that eventually all colonies that are not singletons have positive $S$-count. To see this, we observe that by induction, after any number of steps, it stays true that a non-singleton $g$ will have positive $S$-count, since no two singleton colonies will ever merge in a cardinality merge. In particular, once a singleton colony $g$ fuses into another colony, it will have positive $S$-count. Since there is a non-singleton colony at bounded distance, this happens after a bounded number of steps.
\end{proof}

\begin{lemma}
\label{lem:Eventually}
Let $G$ be a finitely-generated group, suppose $S \Subset G$, and that $X \subset A^G$ a subshift which is $S^{-1}S$-aperiodic. Suppose further that $C \subset X$ is a closed set such that no point in $C$ is totally periodic and $S$ is safe for $C$. Then then there is a preprocessing step for the process such that the following holds: for all $r, n$, there exists $t$ such that if we run the process on any $x \in C$ for $t$ steps (after the preprocessing step), then every $g$ such that $gx \in C$ has all of $Sg$ in its colony, and whenever $hgx \in C$, the colony of $hg$ has $S$-count of size at least $n$.
\end{lemma}

\begin{proof}
We use the preprocessing from the previous lemma. So every colony already has positive $S$-count in all of $G$, for any $x \in X$. We observe that whenever two colonies merge, the $S$-count of the resulting colony is at least the sum of $S$-counts of the two colonies. Thus, the $S$-count of a colony tends to infinity if and only if the size of the colony tends to infinity.

Let $f_t'(x)$ be the minimal $S$-count of a colony of a cell in $\mathcal{B}_r$, when running the process for $t$ steps on $x \in X$. This tends to infinity for all $x \in C$: Since $C$ has no totally periodic points, the sizes of colonies must grow indefinitely by Lemma~\ref{lem:Stable}. In particular, their $S$-counts tend to infinity by the previous paragraph.

Then the functions $f_i(x) = 1/f_i'(x)$ are a pointwise decreasing sequence of continuous functions tending to the zero function on the clopen set $C$. Since the zero function is continuous, by Dini's theorem the convergence is uniform on the compact set $C$. This means precisely that for some $t$, $f_t'(x) \geq n$ for all $x \in C$. In other words, if we run the process for $t$ steps on any $x \in C$, the colonies of all $g \in \mathcal{B}_r$ have $S$-count at least $n$.
\end{proof}

\begin{example}
Possibly the $S^{-1}S$-aperiodicity assumption on $X$ can be removed, and it suffices that $C$ has no totally periodic points and $S$ is safe for it. However, this would certainly require a more complicated argument. We outline how even the case $G = \Z^2$ can be problematic.

Consider the subshift $X \subset \{0,1,2\}^{\Z^2}$ where the forbidden patterns state that no two $2$s may appear vertically or horizontally adjacent. Consider the clopen set $C = [p]$ where $p \in \{2\}^{\{(0,0)\}}$. This clopen set is $S$-safe for $S = \{(0,0), (1,0), (0,1)\}$. Consider now a configuration $x$ where $2$s appear precisely on $\{0\} \times 2\Z$, and which is periodic under the action of $\{0\} \times 2\Z$. Certainly, any such configuration is in $X$.

In the preprocessing set, we would certainly want to perform $S$-merges at every $C$-occurrence. When applied to $x$, the cells in $(\{0,1\} \times 2\Z) \cup (\{0\} \times \Z)$ are now in non-singleton colonies, and other cells are in singleton colonies. We now explain why it is possible to continue the process so as to make all $S$-counts tend to infinity, and in turn so that they do not tend to infinity.

First, let us explain how to achieve this if we use a process that does not use all possible triples $(k, h, p)$ and need not alternate cardinality and pattern merges. To have the $S$-counts tend to infinity, first perform a cardinality merge with $h = (0,-1)$ so that the non-singleton colonies are precisely the boxes $\{2n,2n+1\} \times \{0,1\}$. Then keep repeating cardinality merges only in directions $(1,0)$ and $(-1,0)$ with $k = 1$. The boxes become elongated rectangles, which clearly have large and larger $S$-counts.

To have the number of $S$-counts not tend to infinity. Use pattern matches, and organize the content of $x$ so that the non-singleton colonies always form a thick column $\llb r, r \rrb \times \Z$, and we always match the pattern only at the frontier of this column. We ensure that the vertical parity of a position is always visible in a finite neighborhood. Then we can ensure that we never fuse two vertically adjacent cells into the same colony at the ``frontier'' of the initial column of boxes that keeps expanding (for example, when expanding to the left, even cells can be fused to the right, and odd cells in direction $(1, d)$ for suitably chosen $d \in \Z$). Then clearly the $S$-counts do not increase.

To achieve the same with the existing process where all $(k,h,p)$ are cycled through should certainly be possible, but working through the details would be an exercise in futility.

On an abelian group, it should not be difficult to avoid these issues and ensure that the $S$-aperiodicity is transported everywhere. On a general group, the issue is more serious. \qee
\end{example}

\begin{remark}
The author invented this process to prove the marker lemma for the paper \cite{Sa12b} (this was on an abelian group, where one can use e.g.\ the arguments from \cite{PoSa24} to expand the set of brains, as soon as they form a sparse enough set). Once the author realized they could just use the standard marker lemma in \cite{Sa12b}, this idea was buried, but it turned out to be precisely what we needed for the present situation.
\end{remark}

\section{Permutations under a clopen set}

In this section, we define the subgroup of $\Aut(X \times Y)$ which is not going to be residually finite when $X$ does not have dense periodic points, and $X, Y$ have some additional dynamical properties related to aperiodicity of $X$, and exchangeable patterns in $Y$. Furthermore, these are interlinked so that when $Y$ is a full shift, the aperiodicity assumptions on $X$ become trivial.

Let $Y \subset B^G$ be a subshift. A set of patterns with the same domain are \emph{exchangeable} if in any context, one can be replaced by another without introducing any forbidden patterns. In a formula, $P \subset A^D$ is exchangeable if for all $x \in B^{G \setminus D}$, $x \sqcup p \in Y \iff x \sqcup q \in Y$ for all $p, q \in P$. Exchangeability is an equivalence relation, which we denote by $\equiv$, so for each $D$ we have a partition of $Y|D$ into maximal exchangeable sets. We say a permutation of $Y|D$ \emph{respects} this equivalence relation, if it stabilizes every equivalence class, and we write the group of permutations that respect it as $\Sym_\equiv(Y|D)$. Write $\Alt_\equiv(Y|D)$ for the permutations that are even in each equivalence class separately (note that this is a proper subgroup of $\Alt(Y|D) \cap \Sym_\equiv(Y|D)$ unless there is only one nontrivial equivalence class).

Usually, we want $Y$ to have large exchangeability classes. We say a subshift $Y$ has the \emph{local many fillings property} of \emph{LMFP} if for some $D \Subset G$, the exchangeability class of every pattern in $Y|D$ has cardinality at least $2$.

\begin{lemma}
\label{lem:LMFP}
If $Y$ has the LMFP, then for all $n$, for some $D \Subset G$, the exchangeability class of every pattern in $Y|D$ has cardinality at least $n$.
\end{lemma}

\begin{proof}
Denote by $m(D)$ the  minimal size of an equivalence class for a pattern with domain $D$. If $D \cap Dg = \emptyset$, it is easy to see that $m(D \cup gD) \geq 2m(D)$ (note in particular that this does not require any further mixing assumptions). The proof follows easily by induction.
\end{proof}

For $Y = k^G$ a full shift, the exchangeability relation is full, and every permutation in $\Sym(Y|D) = \Sym(k^D)$ respects exchangeability.


\begin{definition}
\label{def:fupi}
Let $G$ be a group, let $X \subset A^G$, $Y \subset B^G$ be subshifts, let $U \subset X$ be clopen, let $e \in S \Subset G$, let $\pi \in \Sym(Y|S)$ respect exchangeability. Suppose that $S$ is safe for $U$, i.e.\ $tU \cap U = \emptyset$ for any $t \in S^{-1}S \setminus \{e\}$. Then we define $f_{U, \pi} : X \times Y \to X \times Y$ as follows:
\[ f_{U, \pi}(x, y)_g = \begin{cases}
(x_g, \pi(gy|S)_s) & \mbox{if } s \in S, s^{-1}gx \in U \\
(x_g, y_g) & \mbox{otherwise}.
\end{cases} \]
\end{definition}

\begin{example}
\label{ex:fupi}
Let $X \subset \{0,1\}^\Z$ be the \emph{golden mean shift}, with single forbidden pattern $11$. Then $S = \{0,1\}$ is safe for the clopen set $U = [1]$. Let $Y \subset \{0,1,2\}^\Z$ be the subshift with the single forbidden pattern $\{00\}$. Then the permutation $\pi = (12; \; 21)$ is safe for $Y$. An application of $f_{U, \pi}$ is shown in Figure~\ref{fig:fupi}.
\end{example}

\begin{figure}
\centering
\begin{tikzpicture}[scale=0.55]
\draw (-2,0) grid (10,2);
\draw[dashed] (-2,0) -- (-3,0);
\draw[dashed] (-2,1) -- (-3,1);
\draw[dashed] (-2,2) -- (-3,2);
\draw[dashed] (10,0) -- (11,0);
\draw[dashed] (10,1) -- (11,1);
\draw[dashed] (10,2) -- (11,2);

\node () at (-1.5,4.5) {$0$};
\node () at (-0.5,4.5) {$0$};
\node () at (0.5,4.5) {$0$};
\node () at (1.5,4.5) {$1$};
\node () at (2.5,4.5) {$0$};
\node () at (3.5,4.5) {$1$};
\node () at (4.5,4.5) {$0$};
\node () at (5.5,4.5) {$0$};
\node () at (6.5,4.5) {$1$};
\node () at (7.5,4.5) {$0$};
\node () at (8.5,4.5) {$0$};
\node () at (9.5,4.5) {$0$};

\node () at (-1.5,3.5) {$1$};
\node () at (-0.5,3.5) {$2$};
\node () at (0.5,3.5) {$0$};
\node () at (1.5,3.5) {$1$};
\node () at (2.5,3.5) {$2$};
\node () at (3.5,3.5) {$2$};
\node () at (4.5,3.5) {$1$};
\node () at (5.5,3.5) {$0$};
\node () at (6.5,3.5) {$1$};
\node () at (7.5,3.5) {$1$};
\node () at (8.5,3.5) {$0$};
\node () at (9.5,3.5) {$1$};

\draw[line width=2] (1,4) rectangle (2,5);
\draw[line width=2] (3,4) rectangle (4,5);
\draw[line width=2] (6,4) rectangle (7,5);

\node () at (4,2.5) {$\Downarrow$};

\draw (-2,3) grid (10,5);
\draw[dashed] (-2,3) -- (-3,3);
\draw[dashed] (-2,4) -- (-3,4);
\draw[dashed] (-2,5) -- (-3,5);
\draw[dashed] (10,3) -- (11,3);
\draw[dashed] (10,4) -- (11,4);
\draw[dashed] (10,5) -- (11,5);
\node () at (-1.5,1.5) {$0$};
\node () at (-0.5,1.5) {$0$};
\node () at (0.5,1.5) {$0$};
\node () at (1.5,1.5) {$1$};
\node () at (2.5,1.5) {$0$};
\node () at (3.5,1.5) {$1$};
\node () at (4.5,1.5) {$0$};
\node () at (5.5,1.5) {$0$};
\node () at (6.5,1.5) {$1$};
\node () at (7.5,1.5) {$0$};
\node () at (8.5,1.5) {$0$};
\node () at (9.5,1.5) {$0$};

\node () at (-1.5,0.5) {$1$};
\node () at (-0.5,0.5) {$2$};
\node () at (0.5,0.5) {$0$};
\node () at (1.5,0.5) {$2$};
\node () at (0.5,0.5) {$0$};
\node () at (1.5,0.5) {$2$};
\node () at (2.5,0.5) {$1$};
\node () at (3.5,0.5) {$1$};
\node () at (4.5,0.5) {$2$};
\node () at (5.5,0.5) {$0$};
\node () at (6.5,0.5) {$1$};
\node () at (7.5,0.5) {$1$};
\node () at (8.5,0.5) {$0$};
\node () at (9.5,0.5) {$1$};

\draw[line width=2] (1,0) rectangle (3,1);
\draw[line width=2] (3,0) rectangle (5,1);
\draw[line width=2] (6,0) rectangle (8,1);

\end{tikzpicture}
\caption{Application of $f_{U, \pi}$ from Example~\ref{ex:fupi}. Occurrences of $U$ (i.e.\ $1$-symbols) are highlighted in the preimage. The corresponding areas where $\pi$ was applied are highlighted in the image. Of course, $11$ is fixed by $\pi$, so the pattern stays fixed.}
\label{fig:fupi}
\end{figure}


\begin{lemma}
\label{lem:Homomorphism}
Under the assumptions of the definition, the map $f_{U, \pi}$ is an automorphism of $X \times Y$. Furthermore, if $U$ and a safe set $S$ are fixed, then the map from $\Sym_\equiv(Y|S) \to \Aut(X \times Y)$ defined by $\pi \mapsto f_{U, \pi}$ is an embedding.
\end{lemma}

\begin{proof}
Clearly the map $f_{U, \pi}$ is shift-commuting and continuous, as it is defined by a local rule (though a priori its codomain might be $X \times B^G$). Consider $x, y \in X \times Y$. Let $R \subset G$ be the set of elements $g \in G$ such that $gx \in U$. The action of $f_{\pi, U}$ does not modify $x$ in any way. In $y$, it affects at most the cells at $g \in G$ such that $s^{-1}g \in R$ for some $s \in S$. Observe that this cannot happen for two different $s_1, s_2 \in S$, since if $s_1^{-1}g, s_2^{-1}g \in R$, then $s_1^{-1} gx \in U \cap t U$ where $t = s_1^{-1} s_2 \in S^{-1} S$.

Now the action of $f_{U,\pi}$ is simply to apply $\pi$ to the pattern $gy|S$ whenever $g \in R$, but seen as a relative pattern at $g$, i.e.\ we rewrite the pattern $y|Sg$. This is well-defined because the areas affected are disjoint for distinct nodes $g, g' \in R$, and the patterns are exchangeable. Note here that while exchangeability is a property of individual exchanges of patterns, we can also do infinitely many changes at once, as this can be seen as a limit of a sequence of exchanges (and, of course, $Y$ is closed). So the codomain is in fact $X \times Y$.

This description also shows that $\pi \mapsto f_{U,\pi}$ is a homomorphism, which shows that $f_{U,\pi}$ is bijective, as $f_{U,\pi^{-1}}$ is its inverse map (so $f_{U,\pi} \in \Aut(X \times Y)$). It is an embedding because any $Y|S$ can appear under an occurrence of any $U \subset X$, since we work with the full product subshift $X \times Y$.
\end{proof}

\section{Non-residual finiteness of $\Aut(X \times Y)$}


\begin{theorem}
\label{thm:NonRF}
Suppose $G$ is a finitely-generated group. Suppose $Y \subset B^G$ has LMFP. Then there exists $D$ such that if
\begin{itemize}
\item $X \subset A^G$ is a $D^{-1} D$-aperiodic subshift, and
\item $X$ has an open set $U$ such that no point in $U$ has finite orbit, then 
\end{itemize}
the group $\Aut(X \times Y)$ is not residually finite. If $Y$ is a full shift, we may take $D = \{e\}$.
\end{theorem}

\begin{proof}
We first prove the general claim, and explain the minor modifications needed for the last sentence at the end. Let $D \Subset G$ be such that each exchangeability class in $Y|D$ has cardinality at least $3$. Such $D$ exists by Lemma~\ref{lem:LMFP}. Let $U \subset X$ be a clopen set containing no periodic points. Note that by refining $U$, $D$ automatically becomes safe for $U$. Otherwise, by taking a limit we find a point in $X$ with a period in $D^{-1}D \setminus \{e\}$, which was assumed not to exist. We may assume $U$ is a cylinder (by taking any cylinder inside it), and then by recoding $X$ we may assume it is a basic cylinder, i.e.\ of the form $U = [a] = \{x \in X \;|\; x_e = a\}$. Define $f = f_{U, \hat\pi}$ for some nontrivial even permutation $\hat\pi \in \Alt_\equiv(Y|D)$. 

We will now show that for all $n \in \N$,
\begin{itemize}
\item there exist finite simple subgroups $\mathcal{G}_i$, which commute pairwise and are all of cardinality at least $n$,
\item $f$ can be written as a product of automorphisms $f_i$ which belong to the groups $\mathcal{G}_i$, respectively.
\end{itemize}

This implies that $\Aut(X \times k^G)$ is not residually finite: If it were, it would admit a finite quotient $Q$ where $f$ maps nontrivially. Let then $n > |Q|$ and write $f$ in the form described in the second item. Then since $f$ is a product of the $f_i$, some $f_i$ maps nontrivially into $Q$ as well, but then the simple group $\mathcal{G}_i$ embeds into $Q$, contradicting $|\mathcal{G}_i| \geq n > |Q|$.

To prove the claim, we apply Lemma~\ref{lem:Eventually} to the clopen set $U$, starting with $D$-merges at each occurrence of $U$. Specifically, we obtain $t$ such that after $t$ time steps of applying the process,
\begin{itemize}
\item for every $g \in G$ such that $gx \in U$, $Dg$ is contained in the colony, and
\item every colony $S$ of any $g$ such that $gx \in U$ contains at least $n$ many $h \in G$ such that $Dh \subset S$.
\end{itemize}

The set $V$ of all points $x$ where the colony of $e$ is of size at least $n$ after $t \geq 1$ steps of the process is itself of course clopen, and by the previous paragraph it contains $U$. Concretely this means that $V$ contains every configuration $x$ with $x_e = a$. The set $V$ admits a natural finite decomposition according to the precise shape of the set $S$ relative to the brain, and finally the subset of $S$ where the symbol $a$ appears. We write these sets as $V_{S, T, g}$, and define $V_{S, T} = \bigcup_{g \in S} V_{S, T, g}$. More precisely, if after applying the process to $x \in X$ for $t$ steps, the colony of $e$ has its brain at $h$, the colony of $e$ is $S'$, and when $k \in S'$ we have $x_k = a$ precisely if $k \in T'$, then $x \in V_{S, T, g}$ where $g = h^{-1}$ and $S = S'h^{-1}$, $T = T'h^{-1}$. Note that the process starts with a $D$-merge at all occurrences of $a$, so we have $DT \subset S$.

For each of the finitely many possible (brain-relative) shapes $S$ and sets $T$, the map $\pi \mapsto f_{V_{S, T}, \pi}$ from $\Alt_\equiv(Y|S)$ to $\Aut(X \times Y)$ is an embedding by Lemma~\ref{lem:Homomorphism}. Furthermore, for each $S, T$, the group $\Alt_\equiv(Y|S)$ is a product of alternating groups. Since the sets $S$ eventually contain many right translates of $D$, eventually they contain many disjoint translates, so the equivalence classes in $Y|S$ become arbitrarily large. This implies that $\Alt_\equiv(Y|S)$ is a product of alternating groups on arbitrarily large sets, thus of arbitrarily large simple groups.

List all these groups as $\mathcal{G}_i$, specifically each $i$ corresponds to a choice of $S, T$, and an equivalence class $E$ of $Y|S$, and $\mathcal{G}_i$ is isomorphic to the alternating group on $E$ (and specifically acts under occurrences of $V_{S, T}$). Think now about the action of $f$ under an occurrence of $V_{S, T}$. Clearly, it simply applies $\hat\pi$ at each $Dg$ with $g \in T$. Let $f_i \in \mathcal{G}_i$ be the automorphism that applies $\hat\pi$ in each position of $T$ under occurrences of $V_{S, T}$, whenever $Y|S \in E$. Then the product of the $f_i$ over all possible $i$ is precisely $f$. This concludes the proof of the general case.

In the case of a full shift, we need a minor change: if we take $D = \{e\}$, then the cardinality of an equivalence class may not be $3$, specifically if $k = 2$. We describe a uniform way to deal with all full shifts. Take $f = f_{U, \hat\pi}$ where $\hat\pi$ has the same parity as $k$. If the alphabet is even, then $\hat\pi$ is even when seen as a permutation of one coordinate of a power of the alphabet, and if $k$ is odd then $\hat\pi$ is directly an even permutation of the alphabet. In either case, as in the proof of the general case, under an occurrence of $V_{S, T}$ with $|S| \geq 2$, it applies an even permutation of $\llb k \rrb^S$ (it applies $\hat\pi$ in positions $g \in T$), so we can follow the same proof.
\end{proof}

\begin{corollary}
Let $G$ be a finitely-generated group and suppose $X$ does not have dense totally periodic points. Then the group $\Aut(X \times B^G)$ is not residually finite for $|B| \geq 2$.
\end{corollary}

\begin{proof}
Since periodic points are not dense, there is a clopen set $U$ without any totally periodic points. In the previous theorem, when $Y$ is a full shift we can take $D = \{e\}$, and any subshift is trivially $D^{-1}D$-aperiodic (this is an empty requirement).
\end{proof}

\section{The group $\Z^2$}

\begin{theorem}
If $G$ contains a copy of $\Z^2$, then $G$ admits a strongly irreducible subshift whose automorphism group is not residually finite.
\end{theorem}

\begin{proof}
Let $X \subset A^{\Z^2}$ be Hochman's strongly irreducible subshift with no periodic points from \cite{Ho25}. In particular, periodic points are not dense in $X$, so by Theorem~\ref{thm:NonRF}, $Z = X \times B^{\Z^2}$ has non-RF automorphism group for $|B| \geq 2$, and of course it is still strongly irreducible. We can extend this to $G$ by taking the \emph{free extension} to $G$, meaning the subshift $Z' \subset (A \times B)^G$ where 
\[ z \in Z' \iff \forall g \in G: gz|\Z^2 \in Z \]
which is clearly still strongly irreducible, and whose automorphism group clearly contains a copy of $\Aut(Z)$. 
\end{proof}

We will need the following well-known result. We state a proof in symbolic dynamical language.\footnote{The author learned from Dominik Kwietniak that the trick to the proof is to find a synchronizing word.}

\begin{lemma}
\label{lem:ZSI}
Let $X \subset A^\Z$ be a strongly irreducible subshift. Then $X$ has dense totally periodic points.
\end{lemma}

\begin{proof}
Say a word $w \in A^n$ is \emph{synchronizing} (for $X$) if it appears in some configuration of $X$, and for all $x \in A^{(-\infty,-1]}$ and $y \in A^[n, \infty)$ we have
\[ x \sqcup w \in X|(-\infty, n-1] \wedge w \sqcup y \in X|[0, \infty). \]

Let $r$ be the strong irreducibility radius, meaning we can take $S = \llb -r, r \rrb$ in the definition of strong irreducibility. We show that a synchronizing word exists. For a contradiction, suppose one does not exist. Let $w_1 \in X|\llb 2r+1 \rrb$ be arbitrary. Since $w_1$ is not synchronizing, by compactness we find words $u_1, v_1$ such that
\[ u_1w_1 \sqsubset X, w_1v_1 \sqsubset X, \mbox{ but } u_1w_1v_1 \not\sqsubset X \]
By strong irreducibility, and we find some $w_2 \in X|\llb 2r+1\rrb \setminus \{w_1\}$ such that $u_1w_2v_1 \sqsubset X$. Since this word is not synchronizing, we find words $u_2, v_2$ such that
\[ u_2u_1w_2v_1 \sqsubset X, u_1w_2v_1v_2 \sqsubset X, \mbox{ but } u_2u_1w_2v_1v_2 \not\sqsubset X. \]

Continuing this argument inductively, we find an infinite sequence of \textbf{distinct} words $w_1, w_2, w_3, \ldots \in X|\llb 2r+1 \rrb$ and words $u_i, v_i$ such that
\[ u_k \cdots u_2u_1w_k v_1 v_2 \cdots v_{k-1} \sqsubset X, u_{k-1} \cdots u_2u_1w_k v_1 v_2 \cdots v_k \sqsubset X, \]
but
\[ u_k \cdots u_2u_1w_k v_1 v_2 \cdots v_k \not\sqsubset X. \]
This is a contradiction, since $X|\llb 2r+1\rrb$ is finite.

Once have a synchronizing word $w \in A^n$, we are almost done. By strong irreducibility, we find a configuration $x \in X$ such that $k(n+2r+1) \cdot x|\llb n \rrb = w$ for all $k \in \Z$. Clearly, any two words of the form $wuw$ and of the same length are exchangeable. Thus, the $(n+2r+1)\Z$-periodic point $y \in A^\Z$ with $y|\llb n+2r+1\rrb = x|\llb n+2r+1\rrb$ is in $X$. In particular $X$ contains a periodic point. By \cite{CeCo12}, then totally periodic points are dense (alternatively, one can directly include any word inside a synchronizing word, or inside $u$).
\end{proof}

It is shown in \cite{Li03} that strongly irreducible $\Z^2$ SFTs have dense periodic points, and it is well-known that the same is true for block-gluing $\Z^2$-SFTs. Recall that a $\Z^2$-subshift $X$ is \emph{block gluing} if for some \emph{radius} $r$, for all $x, y \in X$ there exist $z, z' \in X$ such that
\[ z|(-\N \times \Z) = x|(-\N \times \Z), \;\; z|([r, \infty) \times \Z) = y|([r, \infty) \times \Z), \]
\[ z'|(\Z \times -\N) = x|(\Z \times -\N), \;\; z'|(\Z \times [r, \infty)) = y|(\Z \times [r, \infty)). \]
This of course implies by strong irreducibility.

We need a slight strengthening of this result, where the periodic points are found in a containing SFT instead of the SFT directly being strongly irreducible (or block gluing).

\begin{lemma}
Let $X \subset Y \subset A^{\Z^2}$ be subshifts, where $Y$ is SFT, and $X$ is block gluing. Then periodic points of $Y$ are dense in $X$.
\end{lemma}

\begin{proof}
Suppose $X$ is block gluing with radius $r$. Let $p \in A^D$ be any pattern that appears in $X$. We may assume $D = \llb 0, n-1\rrb^2$. Define $S = \Z \times \llb 0,n-1 \rrb$ and consider $X|S$ as a $\Z$-subshift under the horizontal shift. It is easily seen to be strongly irreducible, therefore it has dense periodic points by Lemma~\ref{lem:ZSI}.

Let $s \in X|S$ be periodic, and use block gluing of $X$ to obtain a point $x \in X$ such that
\[ x_{i, j} = s_{i, j} = x_{i, j + n+r} \]
for all $i \in \Z, j \in \llb 0,n-1 \rrb$. In particular this point is in $Y$. Since $Y$ is SFT, we can repeat the segment in between to obtain a point $y \in Y$ with
\[ y_{i, j + k(n+r)} = x_{i, j} \]
for all $i \in \Z, k \in \Z, j \in \llb 0, n+r-1 \rrb$. This point now belongs to the subshift $Z$ of $Y$ consisting of the $(\{0\} \times (n+r-1)\Z)$-periodic points. Of course $Z$ is a $\Z^2$-SFT itself, and under the action of $\Z \times \{0\}$ it is a $\Z$-SFT. This, we can find a periodic point $z$ in it. This point of course then corresponds to a totally periodic of $Y$. Since the original pattern $p$ appears syndetically in $y$, it does so in $z$, and we have found a totally periodic point in $Y$ containing $p$.
\end{proof}

\begin{lemma}
On $\Z^2$, finite subshifts are dense in block gluing subshifts in the Hausdorff metric.
\end{lemma}

\begin{proof}
Recall that for subshifts, the Hausdorff metric measures similarity of the languages, and convergence means precisely that for all $D \Subset G$, the sets of $D$-patterns in the subshifts become equal.

Let $X \subset A^{\Z^2}$ be block gluing. Let $D \Subset \Z^2$ be arbitrary. We need to find a finite subshift whose $D$-patterns are precisely $D$. Take the SFT approximation $Y$ corresponding to $D$, meaning the SFT where a $D$-pattern is forbidden if and only if it does not appear in $X$. Then totally periodic points of $Y$ are dense in $X$. Thus, we can find a finite-index subgroup $L \leq \Z^2$ such that $L$-periodic points of $Y$ already contain all $D$-patterns that appear in $X$. Then the finite subshift $Z \subset Y$ of $L$-periodic points has the same $D$-patterns as $X$: It cannot contain more patterns, since even $Y$ does not. On the other hand, we explicitly chose $L$ so that all $D$-patterns appear.
\end{proof}

\begin{theorem}
Let $X \subset A^{\Z^2}$ be a block gluing subshift. Then $\Aut(X)$ is locally embeddable in finite groups.
\end{theorem}

\begin{proof}
Let $F \Subset \Aut(X)$ be any finite set. We need to find a finite group containing a subset $F'$, and a bijection $\phi : F \to F'$, so that whenever $f, g \in F$ and $f \circ g \in F$, we have $\phi(f) \cdot \phi(g) = \phi(f \circ g)$. Let $r$ be the maximal radius of any automorphism in $F$ as a cellular automaton, or its inverse.

Suppose now that $Z$ is any subshift with the same $B_{2r}$-patterns as $X$. Then the local rules of automorphisms in $F$ define automorphisms of $Z$: If we compute the composition of the local rules of $f \in F$ and $f^{-1} \in F$ in either order, the calculation in $Y$ is the same as in $X$, thus we get the identity map, proving both are automorphisms.

Furthermore, if $f, g \in F$ are distinct, then they are also distinct when interpreted as automorphisms of $Z$. Again by the same argument, if $f, g, f \circ g \in F$ and we calculate the same compositions in $Z$, we obtain the same result. Thus, we have an injective map $\psi : \Aut(X) \to \Aut(Z)$ which satisfies $f, g, f \circ g \in F \implies \psi(f) \circ \psi(g) = \psi(f \circ g)$.

When $X$ is block gluing, by the previous lemma we can take $Z$ finite, thus $\Aut(Z)$ is finite and we have precisely a local embedding of $F$ into it.
\end{proof}

\section{Permutation groups on Cartesian products}
\label{sec:Graphs}

We now lay some groundwork for establishing non-residual-finiteness of finitely-generated subgroups of $\Aut(X)$, by studying permutation groups on products of sets.

We will consider in this section permutation groups acting on a Cartesian product of sets. If we have a product $\prod_{g \in E} A_g$, then for $C, D \subset E$, $\pi \in \Sym(\prod_{g \in C} A_g)$ and $P \subset \prod_{g \in D} A_g$, we denote by $f = \pi|P$ the permutation of $\prod_{g \in E} A_g$ defined as follows:
\[ f(p)_g = \begin{cases}
p_g & \mbox{if } g \notin C, \\
p_g & \mbox{if } g \in C \wedge p|D \notin P, \\
\pi(p|C)_g & \mbox{if } g \in C \wedge p|D \in P.
\end{cases} \]
This can be read as ``apply $\pi$ given $P$''; the interpretation is that we check a condition on some of the coordinates, and if this condition holds, then we apply the permutation to another (disjoint) set of coordinates.

The permutation $\pi|P$ can be seen as an abstract analog of the automorphisms $f_{U,\pi}$ (where we also perform a permutation on one part of the input, when the other part satisfies some condition). However, unlike in the automorphism group case, in this section we will not have a perfect separation of ``control and data'', and when we apply these results to automorphisms groups later, their role is different (and mostly hidden into uses of Lemma~\ref{lem:Derp}).

When $P = \prod_{g \in D} A_g$, we simply identify $\pi \in \Sym(\prod_{g \in C} A_g)$ as the corresponding permutation on $\prod_{g \in E} A_g$. More precisely, we have an embedding (that we suppress in notation) called the \emph{natural embedding} from $\Sym(\prod_{g \in C} A_g)$ to $\Sym(\prod_{g \in E} A_g)$. Its image is called the \emph{natural copy} of $\Sym(\prod_{g \in C} A_g)$ in $\Sym(\prod_{g \in E} A_g)$.

Often, we want to define $\pi$ abstractly as a permutation of a Cartesian product, and use a special $@$-notation to clarify which coordinates of a larger product $\pi$ is applied in. For example if we are considering permutations of $A_1 \times A_2 \times A_3$, with $A_1 = A_2 = A_3 = \{0,1\}$, we might consider the $3$-cycle $\pi = (00 \; 01 \; 10)$ of length-$2$ words, and write $\pi @ 12$ for the permutation of $A_1 \times A_2 \times A_3$ defined by $(\pi @ 12)(a, b, c) = (a', b', c)$ where $a'b' = \pi(ab)$. We use similar conventions with patterns, specifically if $a \in A_g$, we write $a@g$ for the pattern $p \in A_g^{\{g\}}$ with $p_g = a$.

Below, we are considering permutations of a Cartesian product $A \times B \times C$. We may then refer to the coordinates by simply the symbols $A, B, C$ (even if they may technically refer to the same set), and think for example of $AC$ as a formal symbol referring to the two outermost coordinates of $A \times B \times C$.

We will use the following ``hypergraph lemma'' from \cite{BoKaSa17}. To state the lemma, we introduce some terminology. Say a \emph{hypergraph} is a set of \emph{nodes}, together with a set of subsets of those nodes called \emph{hyperedges}. We call it a $k$-hypergraph is all the subsets are of cardinality of $k$. A $2$-hypergraph is called a \emph{graph}. A hypergraph has an \emph{underlying graph} which contains those size-$2$ subsets of nodes that are contained in an edge. We say a hypergraph is \emph{weakly connected} if its underlying graph is connected.

\begin{lemma}
\label{lem:Hypergraph}
Let $G$ be a $3$-hypergraph which is weakly connected. Suppose that $\mathcal{G}$ is a subgroup of $\Alt(V(G))$ which contains, for each edge of $G$, the rotations of those nodes. Then $\mathcal{G}$ contains $\Alt(V(G))$.
\end{lemma}

In \cite{BoKaSa17}, this is shown only for finite graphs, but it remains true for general graphs (as a connected graph is a directed union of its connected finite subgraphs).

\begin{lemma}
\label{lem:ABC}
Let $A, B, C$ be nonempty sets. Let $\mathcal{G}$ be the group generated by the natural copies of $\Alt(A \times B)$ and $\Alt(B \times C)$. Then $\Alt(A \times B \times C) \leq \mathcal{G}$ if and only if either
\begin{itemize}
\item one of the sets $A, C$ is a singleton, or
\item all of the sets have at least two elements, and one has at least three elements.
\end{itemize}
\end{lemma}

\begin{proof}
First we discount the uninteresting cases where one of the sets is a singleton. If $A$ or $C$ is a singleton, then naturally $\mathcal{G}$ is generated by $\Alt(A \times B)$ and $\Alt(B \times C)$ (indeed, equal to one of them). If $A$ and $C$ are not singletons, but $B$ is, then there is no communication between $A$ and $C$, and it is easy to find a permutation of $\Alt(A \times C)$ where the effect on one side depends on the other side. If all the sets have cardinality $2$, then setting up any bijection between $\Z_2$ and each of the sets, we see that all permutations are affine maps, from which it is clear that the full alternating group is not generated.

Suppose now that all sets have size at least $2$, and one has size at least $3$. First, consider the case where $|A| \geq 4$. Consider the graph on nodes $ABC$, and edges those $(abc, a'b'c')$ where some $3$-cycle of $ABC$ containing both words is in $\mathcal{G}$. By Lemma~\ref{lem:Hypergraph} we just need to show this hypergraph is weakly connected. We show that any simultaneous change of the first coordinate and another coordinate is possible, i.e.\ $(abc, a'bc')$ and $(abc, a'b'c)$ are edges in its underlying graph when $a \neq a', b \neq b', c \neq c'$. They clearly connect the underlying graph, so this is sufficient.

We may suppose all the sets $A, B, C$ consist of numbers starting from $0$. We have the permutations $\pi_1 = (00;\; 10;\; 20)@AB$ and $\pi_2 = (11;\; 21;\; 31)@AB$ in $\mathcal{G}$, and we have $\pi_3 = (00;\; 10;\; 11)@BC \in \mathcal{G}$. Then
\[ [\pi_2^{\pi_3}, \pi_1] = (000;\; 300)(100;\; 200)@ABC = (00;\; 30)(10;\; 20)@AB|0@C =: \pi_4. \]

This is simply a calculation, but an intuitive explanation is as follows: since $\pi_2$ does not change the $B$ or $C$ coordinates, the effect of $\pi_3$ on $B \times C$ always cancels when $\pi_2^{\pi_3}$ or $\pi_2^{-\pi_3}$ is applied, and $\pi_2^{\pi_3}$ only affects $A$. Thus, both $\pi_1$ and $\pi_2^{\pi_3}$ only affect $A$. Consider now the effect of the commutator on $abc$. If $b \neq 0$, then $\pi_1$ acts trivially on both applications of the commutator, so the commutator cancels. If $b = 0$, then $\pi_2^{\pi_3}$ acts trivially in both applications of the commutator, unless $c = 0$, and thus if $b = 0 \wedge c \neq 0$, the commutator again cancels. If $bc = 00$, then the commutator does not cancel, and its effect is $[(1 \; 2 \; 3), (0\; 1 \; 2)] = (0 \; 3)(2 \; 1)$.

We now observe that since $\Alt(A \times B)$ is simple, it is the normal closure of $(00; \; 30)(20; \; 10)$. Representing elements of $\Alt(A \times B)$ this way, but applying the operations instead to $\pi_4$, we obtain the group $\Alt(A \times B)|0@C$, in particular for any $3$-rotation $(ab;\; a'b';\; a''b'')$ of $AB$, $(ab0;\; a'b'0;\; a''b''0) \in \mathcal{G}$. Since the role of $0$ is irrelevant, we conclude that the edges $(abc, a'b'c)$ (even ones where possibly $a = a'$ or $b = b'$) are all in the graph described in the first paragraph.

Now consider the effect of left-conjugating the $3$-cycle $(000;\; 100;\; 200)$ with $\pi_5^{\pi_6}$ where $\pi_5 = (10;\; 11;\; 01)@AB, \pi_6 = (10;\; 11;\; 01)@BC$, so that the three words map forward by this map. The result is $(000;\; 101;\; 200)$. Again the roles of the symbols are arbitrary, so we can perform any simultaneous change of the $A$- and $C$-coordinates with a $3$-cycle, showing that edges of type $(abc, a'bc')$ are in the graph. We thus have all the edges we claimed, and the hypergraph lemma applies.

Second, consider case where $|B| \geq 3$. Then 
\begin{align*}
\pi &= [(00;\; 01;\; 02)@AB, (00;\; 01;\; 11)@BC] \\
&= (000;\; 021;\; 020;\; 001;\; 011)@ABC \\
&= (00;\; 21;\; 20;\; 01;\; 11)@BC|0@A
\end{align*}
The logic is precisely as previously. We have a $5$-cycle $(00;\; 21;\; 20;\; 01;\; 11)$ whose normal closure in $\Alt(B \times C)$ is all of $\Alt(B \times C)$. In contexts where $0$ does not appear in the $A$-coordinate, their effects cancel, and so they generate precisely $\Alt(B \times C)|0@A$. 
 
The case $|C| \geq 4$ is symmetric to the case $|A| \geq 4$.

The only cases left are $|A| = 3, |B| = 2, |C| = 2$; $|A| = 2, |B| = 2, |C| = 3$; and $|A| = 3, |B| = 2, |C| = 3$. These we checked in GAP, the code can be found in \cite{PermutationChecks}.
\end{proof}

The following lemma generalizes Lemma 3.2 in \cite{Sa22b}.

\begin{lemma}
\label{lem:Derp}
Let $G$ be a connected undirected graph, and suppose that for each $g \in V(G)$ we have a set $A_g$, such that $|A_g| \geq 2$ for all $g$ and $|A_g| \geq 3$ for some $g$. Suppose that $\mathcal{G} \leq \Alt(\prod_{g \in V(G)} A_g)$ contains, for each edge $\{u, v\} \in E(G)$ and for each $\pi \in \Alt(A_g \times A_h)$, the permutation $\hat\pi$ defined by
\[ \hat\pi(p)_w = \begin{cases}
p_w & \mbox{if } w \notin \{u, v\} \\
\pi(p|\{u, v\})_w & \mbox{otherwise.}
\end{cases} \]
Then $\mathcal{G}$ contains $\Alt(A^{V(G)})$.
\end{lemma}

\begin{proof}
We use Lemma~\ref{lem:ABC} and induction. For any connected set $D$ of two nodes containing $g$ such that $|A_g| \geq 3$, we of course have the natural copy of $\Alt(\prod_{k \in D} A_k)$ in $\mathcal{G}$. If $D \cup \{g\}$ is connected, and $D$ contains $k$ such that $A_k \geq 3$, then $g$ is connected to some $h \in C$, and we can take $A = C \setminus \{h\}, B = \{h\}, C = \{g\}$ in the previous lemma to conclude that the natural copy of $\Alt(\prod_{k \in D \cup \{g\}} A_k)$ is contained in $\mathcal{G}$.
\end{proof}

If all the sets have size $2$, then the lemma does not hold, as all permutations are affine, for any choice of bijections between the $A_g$ and the two-element field. 

\section{Finitely-generated non-residually-finite subgroups of automorphism groups}

\begin{definition}
\label{def:pasi}
Let $G$ be a finitely-generated group and $X \subset A^G$ a subshift. For a decomposition $d : B \cong B_1 \times B_2$ as a direct product (through any bijection), the corresponding \emph{partial shifts} are the automorphisms of $X \times B^G$ defined by
\[ \sigma_{d, h}(x, y, z)_g = (x_g, y_g, z_{hg}) \]
where $x \in X, y \in B_1^G, z \in B_2^G$, and $h \in G$. 
\end{definition}

For different bijections $d : B \cong B_1 \times B_2$, we obtain a different group. Note that if we fix a bijection $d : B \cong B_1 \times B_2$ and see configurations through it, partial shifts for this decomposition $d$ combined with the trivial decomposition $B \cong \{1\} \times B$ together generate also the partial shifts $\sigma'_h(x, y, z)_g = (x_g, y_{hg}, z_g)$. Alternatively, these directly correspond to partial shifts for the bijection $B \cong B_2 \times B_1$ obtained from $d$ by swapping the components of the image.

\begin{lemma}
\label{lem:PartialFG}
Let $G$ be a finitely-generated group and $X \subset A^G$ a subshift. There is a finitely-generated subgroup of $\Aut(X \times B^G)$ containing all partial shifts.
\end{lemma}

\begin{proof}
Clearly we need to only consider finitely many decompositions $B \cong B_1 \times B_2$ (since $B$ is finite), since if two decompositions $d : B \cong B_1 \times B_2, d' : B \cong B_1' \times B_2'$ are isomorphic in the sense that for some bijections $d_1 : B_1 \cong B_1'$ and $d_2 : B_2 \cong B_2'$ we have $d'(b) = (d_1(b_1), d_2(b_2))$, then the corresponding partial shifts are the same. For each decomposition, the partial shifts give a group isomorphic to $G$, thus these groups are finitely-generated. Their join then contains all the partial shifts.
\end{proof}

\begin{lemma}
\label{lem:fupi}
Let $G$ be a finitely-generated group and $X \subset A^G$ a subshift. If $|B| \geq 3$, then there is a finitely-generated subgroup $\mathcal{H}$ of $\Aut(X \times B^G)$ containing all $f_{U,\pi}$ where $\pi$ is an even permutation of $B$ and $U$ is a clopen set.
\end{lemma}

\begin{proof}
Let $\mathcal{H}$ be the group generated by (finitely many generators for) the partial shifts, and all $f_{[a], \pi}$ where $\pi \in \Sym(B)$ and $a \in A$. We observe that for any clopen sets $U, V$ and any $\pi, \pi' \in \Sym(B)$, we have
\[ [f_{U,\pi}, f_{V,\pi'}] = f_{U \cap V, [\pi, \pi']}. \]
Using this, and the fact that the lower central series of $S_n$ terminates in $A_n$ for $n \geq 3$, we can prove by induction on $|D|$ that for all cylinders $[p]$ and patterns $p \in A^D$ the automorphisms $f_{[p], \pi}$ are in $\mathcal{H}$.

For the case $|D| = 1$, we note that $f_{g[a], \pi} \in \mathcal{H}$ even for all $\pi \in \Sym(B)$, by conjugating $f_{[p], \pi}$ with a partial shift. Specifically use the trivial decomposition $B = \{1\} \times B$ so that the entire $B$-track is shifted. For the general case, let $D = D' \sqcup \{g\}$. Then we have in $\mathcal{H}$ the maps $[f_{[p|D'], \pi'}, f_{[p|\{g\}], \pi}] = f_{p, [\pi', \pi]}$ for $\pi' \in \Alt(B)$ and $\pi \in \Sym(B)$. Since the commutators $[\pi', \pi]$ generate $\Alt{B}$, this proves the induction step.

For any clopen set $U$, we can write it as a union of cylinders, and represent $f_{U,\pi}$ as a commuting product $f_{U,\pi} = \prod_i f_{[p_i], \pi}$.
\end{proof}

\begin{theorem}
\label{thm:FGNonRF}
If $G$ is finitely generated and $X \subset A^G$ does not have dense periodic points then $\Aut(X \times B^G)$ has a finitely-generated subgroup $\mathcal{G}$ that is not residually finite if one of the following holds:
\begin{itemize}
\item $|B| = mn$ with $m \geq 2, n \geq 3$, or
\item $|B| \geq 3$ and $X$ is $\mathcal{B}_1$-aperiodic.
\end{itemize}
In the first case, we can use the group $\mathcal{G}$ constructed in the proof of Lemma~\ref{lem:fupi}. 
\end{theorem}

\begin{proof}
Note that in each case $|B| \geq 3$, so Lemma~\ref{lem:fupi} applies. As generators of the finitely-generated subgroup $\mathcal{G}$ we take the generators of $\mathcal{H}$ from Lemma~\ref{lem:fupi}, and in the case that $|B| \geq 3$ and $X$ is $\mathcal{B}_1$-aperiodic, also finitely many additional $f_{U, \pi}$ (where $\pi$ permutes a non-singleton safe set of some clopen set $U$)

Consider first the case $|B| = mn$ with $m \geq 2, n \geq 3$. We realize the scheme from the proof of Theorem~\ref{thm:NonRF} in the group $\mathcal{G}$. We note that the map $f = f_{U, \hat\pi}$ from that proof is directly in $\mathcal{G}$ by Lemma~\ref{lem:fupi}. Recall that here $U$ is the nonempty clopen set containing no totally periodic points (which was assumed to exist), and $\hat\pi \in \Sym(A)$ is any nontrivial permutation whose parity is that of $|A|$.

We now need to show that for a clopen set $V$ with safe set $S$, we can perform any even permutation of $\pi \in \Alt(k^S)$. Again by Lemma~\ref{lem:fupi}, we have $f_{V, \pi} \in \mathcal{G}$ for any $\pi \in \Alt(B^{\{g\}})$. Fix a decomposition $B \cong B_1 \times B_2$ with $|B_1| = m, |B_2| = n$. Using partial shifts $\sigma'_k$ for this decomposition, and $\sigma_k$ for the trivial decomposition $B \cong \{1\} \times B$, we obtain that $f_{V, \pi} \in \mathcal{G}$ whenever $\pi \in \Alt(B_1^{\{g\}} \times B_1^{\{h\}})$, and $g, h \in G$.

In particular, if $S$ is a safe set for $V$, consider the permutations $P$ of $(B_1 \times B_2)^S$ that we can perform in relative positions under an occurrence of $V$. If we construct the complete bipartite graph with $2|S|$ nodes, one corresponding to each $B_1^{\{g\}}$ or $B_2^{\{h\}}$-symbol, then we have in $\mathcal{G}$ the $\Alt$-groups corresponding to all edges. By Lemma~\ref{lem:Derp}, we have $P = \Alt((B_1 \times B_2)^S) = \Alt(B^S)$. We conclude that all the maps $f_i$ that appear in the proof of Theorem~\ref{thm:NonRF} and all the simple groups $\mathcal{G}_i$ are contained in $\mathcal{G}$. Thus, $\mathcal{G}$ is not residually finite.

Consider then the case $|B| \geq 3$, and suppose that $X$ is $\mathcal{B}_1$-aperiodic. In this case, we use the fact that when following the protocol, the safe sets of the clopen sets $V$ (that correspond to various shapes of safe sets relative to the brain) as connected.

Now, we use the $\mathcal{B}_1$-aperiodicity assumption: If we perform a conjugacy on $X$ initially, we may assume that $\mathcal{B}_1$-aperiodicity manifests in no symbol ever appearing next to itself in configurations $x \in X$. Then for any two adjacent positions $\{g, sg\}$ in the safe set $S$ of $V$, we always have in $V$ different symbols at $g$ and $sg$. We may in our generating set include all permutations of $\Alt(B^2)$ that only act under a specific pair of symbols of the first track, seen in relative positions $\{e, s\}$. Conjugating such permutations with $3$-rotations of individual symbols at relative positions $g$, we effectively obtain a copy of the group $\Alt(B^2)$ which only acts in the relative positions $g, sg$ under an occurrence of $V$.

Since $S$ is connected, by Lemma~\ref{lem:Derp}, we can perform any even permutation of $B^S$ under an occurrence of $V$. Again, this means the automorphisms $f_i$ and the simple groups $\mathcal{G}_i$ are contained in $\mathcal{G}$.
\end{proof}


\bibliographystyle{plain}
\bibliography{../../../bib/bib}{}

\end{document}